\documentclass[11pt]{amsart}

\usepackage{amsmath,amssymb,latexsym,soul,cite,mathrsfs}

\usepackage{color,enumitem,graphicx}
\usepackage[colorlinks=true,urlcolor=blue,
citecolor=red,linkcolor=red,linktocpage,pdfpagelabels,
bookmarksnumbered,bookmarksopen]{hyperref}
\usepackage[english]{babel}

\usepackage[left=2.9cm,right=2.9cm,top=2.8cm,bottom=2.8cm]{geometry}
\usepackage[hyperpageref]{backref}

\usepackage[colorinlistoftodos]{todonotes}
\makeatletter
\providecommand\@dotsep{5}
\def\listtodoname{List of Todos}
\def\listoftodos{\@starttoc{tdo}\listtodoname}
\makeatother

\numberwithin{equation}{section}
\newtheorem{theorem}{Theorem}[section]
\newtheorem{proposition}[theorem]{Proposition}
\newtheorem{lemma}[theorem]{Lemma}
\newtheorem{corollary}[theorem]{Corollary}

\newtheorem{remark}{Remark}

\begin{document}

\title[Multiple solutions FOR A CLASS OF QUASILINEAR...]
{Multiple solutions  FOR A CLASS OF QUASILINEAR PROBLEMS WITH DOUBLE CRITICALITY}
\author{KARIMA AIT-MAHIOUT, CLAUDIANOR O. ALVES AND PRASHANTA GARAIN}

\address[Karima Ait-Mahiout]
{\newline\indent  Laboratoire "Th\'eorie du point fixe et Applications",
\newline\indent \'Ecole Normale Sup\'erieure, BP 92,
\newline\indent Kouba, 16006, Algiers, ALGERIA
\newline\indent
Email: {\tt karima\_ait@hotmail.fr} }

\address[Claudianor O. Alves ]
{\newline\indent Unidade Acad\^emica de Matem\'atica
\newline\indent
Universidade Federal de Campina Grande
\newline\indent
58429-970, Campina Grande - PB, Brazil 
\newline\indent
Email: {\tt coalves@mat.ufcg.edu.br} }

\address[Prashanta Garain ]
{\newline\indent Department of Mathematics
	\newline\indent
Ben-Gurion University of the Negev
	\newline\indent
P.O.B. 653
\newline\indent
Beer Sheva 8410501, Israel
\newline\indent
Email: {\tt pgarain92@gmail.com} }

\pretolerance10000

\begin{abstract}
We establish multiplicity results for the following class of quasilinear problems
$$
\left\{
\begin{array}{l}
-\Delta_{\Phi}u=f(x,u) \quad \mbox{in} \quad  \Omega, \\
u=0 \quad  \mbox{on} \quad  \partial \Omega,
\end{array}
\right.
\leqno{(P)}
$$
where $\Delta_{\Phi}u=\text{div}(\varphi(x,|\nabla u|)\nabla u)$ for a generalized N-function $\Phi(x,t)=\int_{0}^{|t|}\varphi(x,s)s\,ds$. We consider $\Omega\subset\mathbb{R}^N$ to be a smooth bounded domain that contains two disjoint open regions $\Omega_N$ and $\Omega_p$ such that $\overline{\Omega_N}\cap\overline{\Omega_p}=\emptyset$. The main feature of the problem $(P)$ is that the operator $-\Delta_{\Phi}$ behaves like $-\Delta_N$ on $\Omega_N$ and $-\Delta_p$ on $\Omega_p$. We assume the nonlinearity $f:\Omega\times\mathbb{R}\to\mathbb{R}$ of two different types, but both behaves like $e^{\alpha|t|^\frac{N}{N-1}}$ on $\Omega_N$ and $|t|^{p^*-2}t$ on $\Omega_p$ as $|t|$ is large enough, for some $\alpha>0$ and $p^*=\frac{Np}{N-p}$ being the critical Sobolev exponent for $1<p<N$. In this context, for one type of nonlinearity $f$, we provide multiplicity of solutions in a general smooth bounded domain and for another type of nonlinearity $f$, in an annular domain $\Omega$, we establish existence of multiple solutions for the problem $(P)$ that are nonradial and rotationally nonequivalent. 
\end{abstract}

%\thanks{Corresponding author: Prashanta Garain}
\thanks{ C. O. Alves was partially
supported by  CNPq/Brazil 304804/2017-7 }

\maketitle

\section{Introduction}

In this paper, we establish the existence of multiple solutions for the following class of quasilinear problems
$$
\left\{
\begin{array}{l}
	-\Delta_{\Phi}u=f(x,u) \quad \mbox{in} \quad  \Omega, \\
	u=0 \quad  \mbox{on} \quad  \partial \Omega,
\end{array}
\right.
\leqno{(P)}
$$
where $\Omega \subset \mathbb{R}^N$, with $N \geq 2$, is a smooth bounded domain, $\Delta_{\Phi}u={\rm div}\,(\varphi(x,|\nabla u|)\nabla u)$ is the $\Phi$ Laplace operator, where $\Phi(x,t)=\int_{0}^{|t|}\varphi(x,s)s\,ds, \varphi: \Omega \times [0,+\infty)  \to [0,+\infty) $ and $f: \Omega \times \mathbb{R}  \to \mathbb{R} $ are continuous functions that satisfy some hypothesis that will be mentioned later on.

Before proceeding further, let us go through some known results associated with the $\Phi$ Laplace equations. In the recent past, the study of such equations concerning the existence theory has been a research topic of considerable attention. Such operator extends the $p$-Laplace operator, the variable exponent $p$-Laplace operator, weighted $p$-Laplace operator, $p,q$-Laplace operator and indeed occur in many physical phenomena.

When $\Phi$ is independent of $x$, solutions of $(P)$ are investigated in the Orlicz-Sobolev space \cite{PKJF} and we refer the reader to Alves, Figueiredo and Santos \cite{TMNA2014},  Fukagai, Ito and Narukawa \cite{FN}, Carvalho,  Silva, Gon\c calves and Goulart \cite{MSGC}, Fukagai and Narukawa \cite{FN2},  Harjulehto and H\"{a}st\"{o}  \cite{HH},   and their  references for the study of such PDEs. When $\Phi$ also depends on $x$, we are led to study the problems in variable exponent Sobolev \cite{die,KR} or Musielak-Sobolev spaces \cite{Hud, Iwona, Mu, PKJF}. Differential equations in variable exponent Sobolev spaces have been studied extensively in the last years,  most part of them involves the $p(x)$-Laplacian operator, see for example, Alves and Barreiro \cite{AlvesBarreiro}, Alves and Ferreira \cite{AlvesFerreira}, Alves and Souto \cite{AlvesSouto}, Alves and R\u{a}dulescu \cite{AlvesR}, Chabrowski and Fu \cite{chabrowki}, Fan and Zhang \cite{FanZhang}, Fan \cite{Fansub-super},  R\u{a}dulescu and Repov\v{s}
\cite{radrep} and their references. However, Differential equations in general Musielak-Sobolev spaces have been studied very little, see for instance, Azroul, Benkirane,  Shimi and Srati \cite{ABSS}, Benkirane and Sidi El Vally \cite{BSEL2}, Fan \cite{Fan}, Liu and Zhao \cite{LH}, Wang and Liu \cite{WL} and the references therein.

In  the present paper, we will apply some recent results involving the Musielak-Sobolev spaces to study the existence of nontrivial solution for the problem $(P)$. Next, we will state our main hypothesis on the functions $\Phi$ and $\varphi$:
\begin{enumerate}
	\item[$(\varphi_1)$] For each $x \in \Omega$, $\varphi(x,.)$ is a $C^1$ function in the interval $(0,+\infty)$.
	\item[$(\varphi_2)$] $\varphi(x,t)$, $\partial_{t}(\varphi(x,t)t)>0$, for $x \in \Omega$ and $t>0$.
	\item[$(\varphi_3)$] There exists $1<p<N<q<p^*$ such that
	\[
	p\leq \frac{\varphi(x,|t|)|t|^{2}}{\Phi(x,|t|)}\leq q,  \,\,\, \mbox{for} \quad x \in \Omega \quad \mbox{and} \quad t \not= 0.
	\]
\end{enumerate}
\noindent Arguing as in  Fukagai, Ito and Narukawa \cite{FN}, it is possible to prove with few modifications that if $\varphi$ satisfies the conditions $(\varphi_1)-(\varphi_3)$, then the function $\Phi$ is a generalized N-function.

The complementary function $\widetilde{\Phi}$ associated with $\Phi$ is given
by the Legendre's transformation, that is,
\begin{equation} \label{tildePhi}
	\widetilde{\Phi}(x,s) = \max_{t\geq 0}\{ st - \Phi(x,t)\}, \quad x \in \Omega  \quad  \mbox{and} \quad s \in \mathbb{R}.
\end{equation}
The functions $\Phi$ and $\widetilde{\Phi}$ are complement of each other and $\widetilde{\Phi}$ is also a generalized N-function. Hereafter, we also assume that for some constant $d_1$,
\begin{enumerate}
	\item[$(\varphi_4)$] $\displaystyle \inf_{x \in \Omega}\Phi(x,1)=d_1>0$.
	\item[$(\varphi_5)$] For each $t_0 \not=0$, there is a constant $c_0>0$ such that
	$$
	\frac{\Phi(x,t)}{t} \geq c_0 \quad \mbox{and} \quad  \frac{\tilde{\Phi}(x,t)}{t} \geq c_0 \quad \mbox{for} \quad t \geq t_0 \quad \mbox{and} \quad x \in \Omega.
	$$
\end{enumerate}

The conditions $(\varphi_1)-(\varphi_5)$ are very important in our approach, because they permit us to conclude that the Musielak-Orlicz space $L^{\Phi}(\Omega)$ and Musielak-Sobolev space $W^{1,\Phi}(\Omega)$ are reflexive and separable Banach spaces, for more details see Section 2.

In a recent paper, Alves, Garain and R\u{a}dulescu \cite{AGR} proved the existence of at least one nontrivial solution for the following prototype problem
$$
\left\{
\begin{array}{l}
	-\Delta_{\Phi}u=f(x,u) \quad \mbox{in} \quad  \Omega, \\
	u=0 \quad  \mbox{on} \quad  \partial \Omega,
\end{array}
\right.
\leqno{(Q)}
$$
provided $\Omega$ is a smooth bounded domain in $\mathbb{R}^N$ with $N\geq 2$, $f$ is a continuous function, $\varphi:\Omega\times[0,+\infty)\to[0,+\infty)$ satisfies the hypothesis $(\varphi_1)-(\varphi_5)$ above (see \cite{AGR} for precise assumptions).

In the present paper, as in \cite{AGR}, $\Omega$ satisfies the following conditions: There are three smooth domains $\Omega_N, \Omega_q, \Omega_p \subset \Omega$
with nonempty interior such that
$$
\Omega=\Omega_N \cup \Omega_q \cup \Omega_p 
$$
and there is $\delta>0$ such that 
$$
(\overline{\Omega_N})_{\delta} \cap (\overline{\Omega_p})_{\delta}=\emptyset.
$$ 
Hereafter, if $A \subset \Omega$, we denote by $A_\delta$ to be the $\delta$-neighborhood of $A$ restricted to $\Omega$, that is,
$$
A_\delta=\{x \in \Omega\,:\, {\rm dist}\,(x,A)<\delta\}.
$$
Associated with the sets  $\Omega_N, \Omega_q$ and $\Omega_p$, there are three continuous functions $\eta_N,\eta_q,\eta_p:\overline{\Omega} \to [0,1]$ satisfying:
$$
\eta_N(x)=1, \quad \forall x \in \overline{\Omega_N},
$$
$$
\eta_p(x)=1, \quad \forall x \in \overline{\Omega_p},
$$
and
$$
\eta_q(x)=1, \quad \forall x \in \Omega_q= {\Omega} \setminus \overline{(\Omega_N \cup \Omega_p)},
$$

$$
\eta_N(x)=0, \;\; \forall x \in (\overline{\Omega_N})^c_{\delta}, \;\; \eta_p(x)=0, \quad \forall x \in (\overline{\Omega_p})^c_{\delta}, \;\; \eta_q(x)>0, \quad \forall x \in (\overline{\Omega_q})_{\delta}, \eta_q(x)=0, \;\; \forall x \in (\overline{\Omega_q})^c_{\delta}
$$ 
and  for some positive constant $c_4$,
$$
\eta_q(x)\leq c_4 \text{dist}(x, \partial(\Omega_q)_{\delta} \cap \Omega_p)^{l}, \quad \forall x \in \overline{\Omega_p} \cap (\Omega_q)_{\delta},
$$
where $l>q$ and $\text{dist}(x, \partial(\Omega_q)_{\delta} \cap \Omega_p)=\inf\{|x-y|\,:\,y \in \partial(\Omega_q)_{\delta} \cap \Omega_p\}$.

Related to the function $f:\overline{\Omega} \times \mathbb{R} \to \mathbb{R}$ we assume that it is a continuous function of  one the the following forms:
$$
f(x,t)=\lambda\eta_N(x)|t|^{\beta-2}te^{\alpha|t|^{\frac{N}{N-1}}}+\mu\tilde{\eta}_{q}(x)g(x,t)+\eta_p(x)(\tau|t|^{\zeta-2}t+|t|^{p^*-2}t), \quad \forall\,(x,t) \in {\Omega} \times \mathbb{R},
\leqno{(f_1)}
$$
or
$$
f(x,t)=\lambda \eta_N(x)|t|^{\beta-2}te^{\alpha|t|^{\frac{N}{N-1}}}+\tilde{\eta}_{q}(x)g(x,t)+\eta_p(x)|t|^{p^*-2}t, \quad \forall\,(x,t) \in {\Omega} \times \mathbb{R},
\leqno{(f_2)}
$$
where $\lambda,\mu,\tau$ are positive parameters, $\alpha>0$, $p^*>\zeta>q>N>p>\frac{N}{2}$,\,$\beta>q$, where $p^*=\frac{Np}{N-p}$,  $g:\overline{\Omega} \times \mathbb{R} \to \mathbb{R}$ and  $\tilde{\eta}_q: \overline{\Omega}\to [0,1]$ are continuous functions such that
$$
\tilde{\eta}_q(x)=1, \quad\,\forall x \in \Omega_q= {\Omega} \setminus \overline{(\Omega_N \cup \Omega_p)}
$$
and
$$
\tilde{\eta}_q(x)=0, \quad \forall\,x \in (\overline{\Omega_q})^c_{\delta/2}.
$$

Related to the function $g$, we assume the following conditions: 
$$
g \quad \mbox{is odd with respect to the second variable }t, \leqno{(g_0)}
$$
$$
g(x,t)=o(|t|^{q_1-1}), \quad \mbox{as} \quad t \to 0, \quad \mbox{uniformly in} \quad x \in (\overline{\Omega_q})_{\delta/2}  \leqno{(g_1)}
$$
for some $q_1>q$ and there is $\theta>q$ such that
$$
0<\theta G(x,t)\leq g(x,t)t, \quad \forall\,x \in (\overline{\Omega}_q)_{\delta/2} \leqno{(g_2)}
$$
where $G(x,t)=\int_{0}^{t}g(x,s)\,ds$, for $t \in \mathbb{R}$.

There exists a constant $c>0$, such that
$$
g(x,t)\geq ct^{q_2-1},\quad \forall t\geq 0,\,\forall x\in\Omega_q, \leqno{(g_3)}
$$
for some $q_2>q_1,$

With these notations, we are ready to mention the last conditions on $\varphi$. If $f$ is the form $(f_1)$, we assume for each $t>0$ the following:
\begin{enumerate}
	\item[$(\varphi_6)$] $\varphi(x,t) \geq t^{N-2}, \quad \mbox{for} \quad x \in \Omega_N \quad \mbox{and} \quad c_1 t^{N-2} \geq \varphi(x,t), \quad x \in \Omega_N \setminus \overline{(\Omega_q)_{\delta}}.$
	\item[$(\varphi_7)$] $ \varphi(x,t) \geq \tau_1(x)t^{q-2}, \quad \mbox{for} \quad x \in (\Omega_q)_{\delta}$ where $\tau_1:\overline{\Omega} \to \mathbb{R}$ is a continuous function satisfying:
	$$
	\tau_1(x)>0, \quad \forall x \in (\Omega_q)_{\delta} \quad \mbox{and} \quad \tau_1(x)=0, \quad \forall x \in  ((\Omega_q)_{\delta})^c.
	$$
	\item[$(\varphi_8)$] $ \tau_2(x)t^{q-2}+c_2t^{p-2} \geq \varphi(x,t) \geq t^{p-2}, \quad x \in \Omega_p$  where $\tau_2:\overline{\Omega_p} \to \mathbb{R}$ is a nonnegative continuous function satisfying:
	$$
	\tau_2(x)\leq c_3 {\rm dist}(x, \partial(\Omega_q)_{\delta} \cap \Omega_p)^{s}, \quad \forall x \in \overline{\Omega_p} \cap (\Omega_q)_{\delta}
	$$ 
	for some  $s>q$ and 
	$$
	\tau_2(x)=0, \quad \forall x \in \overline{\Omega_p} \setminus \overline{(\Omega_q)_{\delta}},
	$$
	for some constants $c_i>0$ with $i=1,2,3.$
\end{enumerate}
Now, if $f$ is the form $(f_2)$, the condition $(\varphi_6)$ is assumed of the following way: 
\begin{enumerate}
	\item[$(\varphi_6)$] $\varphi(x,t) \geq t^{N-2}, \quad \mbox{for} \quad x \in \Omega_N.$
\end{enumerate}

As a model of a function that satisfies the conditions $(\varphi_1)-(\varphi_8)$ is the function \linebreak $\varphi:{\Omega} \times [0,+\infty) \to [0,+\infty)$ defined by
\begin{equation} \label{vaphi0}
	\varphi(x,t)=\eta_N(x)t^{N-2}+\eta_{q}(x)t^{q-2}+\eta_{p}(x)t^{p-2}, \quad \forall\,(x,t) \in {\Omega} \times [0,+\infty)
\end{equation}
and so,
\begin{equation} \label{Phi}
	\Phi(x,t)=\frac{\eta_N(x)}{N}|t|^{N}+\frac{\eta_{q}(x)}{q}|t|^{q}+\frac{\eta_{p}(x)}{p}|t|^{p}, \quad \forall\, (x,t)\in {\Omega} \times \mathbb{R}.
\end{equation}

Motivated by the study made in \cite{AGR}, we intend to prove existence of multiple solutions for the problem $(Q)$ working with the same operator under the nonlinearities $(f_1)$ and $(f_2)$. Here we have two main results involving multiple solutions and their motivation are mentioned below. 

Our first main result is motivated by the study made in Wei and Wu \cite{ZX92}, where the authors showed the existence of multiple solutions for the following class of problems involving the $p$-Laplacian operator
\begin{equation} \label{WeiWu}
	\left\{	\begin{array}{l}
		-	div(|\nabla u|^{p-2}\nabla u)=f(x,u)+\lambda |u|^{p^*-2}u \quad\text{in }\Omega,\\
		u=0 \quad\text{on } \partial\Omega,
	\end{array}
	\right.
\end{equation}
where $\Omega$ is a bounded domain, $\lambda$ is a positive parameter and $f$ is a continuous function with subcritical growth and $p^*=\frac{Np}{N-p}$ for $N > p$. Using a version of an abstract theorem due to Ambrosetti and Rabinowitz \cite{AmbrosettiRabinowitz} that involves the genus theory for $C^{1}$ even functional, it was proved that given $n \in \mathbb{N}$, there is $\lambda_*=\lambda_*(n)>0$ such that problem (\ref{WeiWu}) has at least $n$ nontrivial solutions for $\lambda \in (0, \lambda_*)$. In \cite{SX2003}, Silva and Xavier improved the main results proved in \cite{ZX92}.

Here, we proved a version of the above mentioned result for the problem $(P)$ and the statement of our result is the following.

\begin{theorem} \label{T1}
	Assume $(g_0)-(g_3), (f_1)$ and  $(\varphi_1)-(\varphi_8)$. Then, for each $k\in\mathbb{N}$, there exists positive real numbers $\lambda_k,\mu_k$ and $\tau_k$ such that for $\lambda\geq\lambda_k,\mu\geq\mu_k$ and $\tau\geq\tau_k$, the problem $(P)$ has at least $k$ pairs of nontrivial solutions.
\end{theorem}

Our second result involves the existence of many rotationally nonequivalent and nonradial solutions. We would like to point out that the existence of many rotationally nonequivalent and nonradial solutions was considered in some problems involving the Laplacian operator. Br\'ezis and Niremberg \cite{Brez} proved the existence of nonradial positive solution for the following problem
\begin{equation} \label{GuidasNN}
	\left\{	\begin{array}{l}
		-	\Delta u+u-u^{p}=0 \quad\text{in } \quad D,\\
		u=0 \quad\text{on } \partial D,
	\end{array}
	\right.	
\end{equation} 	
where 
$$
D=\{x \in \mathbb{R}^N\,:\, r < |x| < r+d\}
$$
for some $d>0$. This type of phenomenon is known as symmetry breaking. In \cite{Coffman}, Coffman proved that the number
of nonradial and rotationally nonequivalent positive solutions of (\ref{GuidasNN}) in $D$ tends
to $+\infty$ as $r$ tends to $+\infty$, if $p > 1$ and $N = 2$ or $1 < p < N/(N-2)$ and $N \geq  3$.

Motivated by the above papers, some authors have studied this class of problems. For the subcritical case, we cite the papers of  Li \cite{Li}, Lin \cite{Lin}, Suzuki \cite{Suzuki} and references therein.

Related to the critical case, Wang and Willem \cite{WangWillem} have showed the
existence of multiple solutions for the following problem 
\begin{equation} \label{WW}
	\left\{	\begin{array}{l}
		-	\Delta u=\lambda u+u^{2^*-1} \quad\text{in } \quad \Omega_r,\\
		u=0 \quad\text{on } \partial \Omega_r,
	\end{array}
	\right.	
\end{equation} 
where 
\begin{equation} \label{omegar}
\Omega_r=\{x \in \mathbb{R}^N\,:\,r < |x| < r+1\}.
\end{equation}
The authors proved that
for $0 < \lambda < \pi^2$ and $n \in \mathbb{N}$, there exists $R(\lambda, n)$ such that for $r > R(\lambda, n)$, the
equation (\ref{WW}) has at least $n$ nonradial and rotationally nonequivalent solutions.
Motivated by \cite{WangWillem}, de Figueiredo and Miyagaki \cite{DjairoOlimpio} have considered
the following problem
\begin{equation} \label{FO}
	\left\{	\begin{array}{l}
		-	\Delta u=f(|x|,u)+u^{2^*-1} \quad\text{in } \quad \Omega_r,\\
		u=0 \quad\text{on } \partial \Omega_r,
	\end{array}
	\right.	
\end{equation} 
where $f$ is a $C^1$ function with subcritical growth.

In \cite{AlvesFreitas}, Alves and de Freitas showed the existence of many rotationally nonequivalent and nonradial solutions for a large class of quasilinear problems that have in particular case the problem below
$$
\left\{
\begin{array}{l}
	-\Delta_{N}u=\lambda|u|^{\beta-2}\beta e^{\alpha|u|^\frac{N}{N-1}} \quad \mbox{in} \quad  \Omega_r, \\
	u=0 \quad  \mbox{on} \quad  \partial \Omega_r.
\end{array}
\right.
\leqno{(R_1)}
$$

Still related to this class of problem, we would like to cite the papers of 
Byeon \cite{Byeon}, Castro and Finan \cite{CAstroFinan}, Catrina and Wang \cite{CatrinaWang},
Mizoguchi and Suzuki \cite{MIzoguchiSuzuki}, Hirano and Mizoguchi \cite{HiranoMizoguchi} and references
therein.

Motivated by bibliography cited above and more precisely, by results found in \cite{AlvesFreitas}, \cite{DjairoOlimpio} and \cite{WangWillem}, we are ready to state our second main result, however we need to fix some more conditions: 
\begin{enumerate}
\item[$(\varphi_{9})$] $\dfrac{\varphi(x,t)}{|t|^{q-3}t}$ is nonincreasing for $t\neq 0$.

\item[$(\varphi_{10})$] $\Phi$ is radial in relation with $x$ that is $\Phi(|x|,t)=\Phi(x,t)$ for all $t>0$.

\item[$(\varphi_{11})$] There exists  $\kappa \in (0, \frac{1}{2^{N+1}})$ such that 
$$
|\partial_s \Phi(s,t)| \leq \kappa \Phi(s,t), \, \forall (s,t) \in \mathbb{R}^2.
$$

\item[$(\eta)$] The functions $\eta_N,\tilde{\eta}_q,\eta_p$ and $g$ are radial in $x$, that is 
$$
\eta_N(x)=\eta_N(|x|),\,\tilde{\eta}_q(x)=\tilde{\eta}(|x|),\,\eta_p(x)=\eta_p(|x|)
$$
and
$$
g(x,t)=g(|x|,t),
$$
for all $x\in\Omega$ and $t>0$.
\item[$(\Omega_N)$] There is $\delta_1>0$ such that
$$
\mathcal{A}=\left\{x \in \mathbb{R}^N\,:\, \frac{2r+1}{2}-\delta_1 \leq  |x| \leq \frac{2r+1}{2}+\delta_1\right\} \subset \Omega_N\setminus\overline{(\Omega_q)_\delta}.
$$ 

\item[$(g_4)$] $g(|x|,.)$ is a $C^1$ function in the interval $(0,+\infty)$ and  $\dfrac{g(|x|,t)}{|t|^{q-1}}$ is increasing for $t\neq 0$ uniformly in $x\in (\overline{\Omega_q})_{\frac{\delta}{2}}$. 

\noindent The reader is invited to see that $\Phi$ given in (\ref{Phi}) also satisfies $(\varphi_9)-(\varphi_{11})$, provided $(\eta)$ holds.

\vspace{0.5 cm}

Our second main theorem has the following statement.
\end{enumerate}
\begin{theorem}\label{T2}
	Assume $\Omega=\Omega_r$ with $N\geq 2$ and $N\neq 3$. Let $(f_2), (g_1),(g_2),(g_4)$, $(\eta)$, $(\Omega_N)$ and $(\varphi_1)-(\varphi_{11})$ holds. Then, for each $n\in\mathbb{N}$, there exists $r_0=r_0(n)>0$ and $\lambda_0=\lambda_0(n)>0$ such that for $\lambda\geq \lambda_0$ and $r\geq r_0$, the problem $(P)$ has at least $n$ nonradial, rotationally nonequivalent and nontrivial solutions.
\end{theorem}

\subsection{Our approach: }
To prove our main results (Theorem \ref{T1}-\ref{T2}), we use variational methods. More precisely, for the proof of Theorem \ref{T1}, we follow the approach from Wei and Wu \cite{ZX92} and Silva and Xavier \cite{SX2003}. To this end, we use a result from Ambrosetti and Rabinowitz (see Lemma \ref{AmRbthm}). To obtain Theorem \ref{T2}, we adapt for our problem some ideas found in de Figueiredo and Miyagaki \cite{DjairoOlimpio} and Alves and de Freitas \cite{AlvesFreitas}. Here, we establish a Poincar\'e type inequality (Lemma \ref{poincare}) and a Strauss type result (Lemma \ref{strauss}) in the setting of Musielak-Sobolev spaces.

It is worth mentioning that, due to the double critical behavior, the energy functionals associated with the problem $(P)$ do not satisfy the $(PS)$-condition at some levels, which brings some difficulties to apply variational methods. To overcome such difficulties, we closely follow the approach introduced in \cite{AGR}, where one needs to simultaneously employ the concentration compactness Lemma due to Lions in $W^{1,p}(\Omega_p)$ found in Medeiros \cite{Everaldo}, see Lemma \ref{concentration},  to obtain a useful estimate related to the critical exponent problem and a version of the Trundiger-Moser inequality in $W^{1,N}(\Omega_N)$ by Cianchi \cite{Cianchi}, see Lemma \ref{TM},  to deal with the exponential growth. Another difficulty appears, since the trace of the functions on $\partial\Omega_p$ and $\partial\Omega_N$ may not vanish. We tackle this difficulty by applying the type of results that are used in the study of Neumann boundary value problems (Lemma \ref{TM2}-\ref{TM4}).

\subsection{Organization of the article}
This article is organized as follows: In Section 2, we discuss some preliminary results for the Musielak-Sobolev spaces, while in Section 3, we show some technical results that will be used in our approach. In sections 4 and 5, we discuss some preliminaries required to prove our main results and finally, in Section 6, we prove our main results (Theorems \ref{T1} and \ref{T2}). 

\subsection{Notations}
Throughout the paper, for $t>1$, we denote by $t'=\frac{t}{t-1}$. By $C$, we mean a constant which may vary from line to line or even over the same line. If $C$ depends on the parameters $r_1,r_2,\cdots,r_k$, we write $C=C(r_1,r_2,\cdots,r_k)$.

\section{A brief review about the Musielak-Sobolev spaces}
In this section, we recall some results on Musielak-Orlicz and Musielak-Sobolev spaces. For more details, we refer to \cite{Iwona,Fan, HH, Mu} and their references.

Let $\Omega \subset \mathbb{R}^N$ be a smooth bounded domain and $\Phi(x,t)=\int_{0}^{|t|}\varphi(x,s)s\,ds$ be a generalized N-function, that is, for each $t \in \mathbb{R}$, the function $\Phi(., t)$ is measurable and for a.e. $x \in \Omega$, the function $\Phi(x, .)$ is an N-function. For the reader's convenience, we recall that a continuous function $A : \mathbb{R} \rightarrow [0,+\infty)$ is an N-function if
\begin{enumerate}
	\item[$(i)$] $A$ is convex.
	\item[$(ii)$] $A = 0 \Leftrightarrow t = 0 $.
	\item[$(iii)$] $\displaystyle\lim_{t\rightarrow0}\frac{A(t)}{t}=0$ and $\displaystyle\lim_{t\rightarrow+\infty}\frac{A(t)}{t}= +\infty$ .
	\item[$(iv)$] $A$ is even.
\end{enumerate}

The Musielak-Orlicz space $L^{\Phi}(\Omega)$ is defined by
\[
L^{\Phi}(\Omega)=\left\{
u:\Omega \to  \mathbb{R} \left\vert \,u\text{ is
	measurable and }\,\,\exists \, \tau >0 \,\, \mbox{such that} \,\, \int_{\Omega}\Phi\left(x,\frac{|u|}{\tau}\right)\,dx<+\infty\right.  \right\}
\]
endowed with the Luxemburg norm
\[
\left\vert u\right\vert _{\Phi}=\inf\left\{  \lambda>0\left\vert
\,\int_{\Omega}\Phi\left(x,\frac{|u|}{\lambda}\right)\,dx \leq1\right.  \right\}  \text{.}
\]

We say that an N-function $\Phi$ satisfies the $\Delta_{2}$-condition, denote by $\Phi \in \Delta_{2}$, if there exists a constant $K>0$ such that
\begin{equation}\label{delta2}
\Phi(x,2t) \leq K\Phi(x,t)\quad \mbox{for} \quad x \in \Omega \quad \mbox{and} \quad t\in \mathbb{R}.
\end{equation}
Arguing as in \cite[Theorem 4.4.4]{PKJF}, it follows that $\Phi$ satisfies  the  $\Delta_{2}$-condition if and only if,
$$
\sup_{(x,t) \in \Omega \times (0,+\infty)}\frac{\varphi(x,|t|)|t|^2}{\Phi(x,|t|)}<+\infty.
$$

Moreover, an important inequality involving $\Phi$ and its complementary function $\tilde{\Phi}$ ( see (\ref{tildePhi}))  is a Young's type inequality given by
\begin{equation} \label{YIne}
st \leq \Phi(x,s) + \widetilde{\Phi}(x,t), \quad x \in \Omega  \quad  \mbox{and} \quad\,\forall s, t\geq 0.
\end{equation}
Using the above inequality, it is possible to prove a H\"older type inequality, that is,
\[
\Big| \int_{\Omega}uvdx \Big| \leq 2 \Vert u \Vert_{\Phi}\Vert v \Vert_{\widetilde{\Phi}}\quad \forall \,u \in L^{\Phi}(\Omega) \quad \mbox{and} \quad \forall \,v \in L^{\widetilde{\Phi}}(\Omega).
\]

Arguing as in  \cite{FN}, if $(\varphi_3)$ holds, we derive that
$$
\frac{q}{q-1}\leq \frac{\tilde{\varphi}(x,|t|)|t|^2}{\tilde{\Phi}(x,|t|)} \leq \frac{p}{p-1}, \quad x \in \Omega \quad \mbox{and} \quad t \not=0,
$$
where
$$
\tilde{\Phi}(x,t)=\int_0^{|t|}\tilde{\varphi}(x,s)s\,ds,
$$
and 
$$
\tilde{\varphi}(x,s)=\sup\{t\,:\,\varphi(x,t)t \leq s\}, \quad x \in \overline{\Omega} \quad \mbox{and} \quad s\geq 0.
$$
Hence, if ($\varphi_3$) holds, we have $\tilde{\Phi}$ also satisfies the $\Delta_{2}$-condition. 

Arguing as in \cite[Lemma A2]{FN}, it is possible to prove that $\Phi$ and $\tilde{\Phi}$ satisfy the following inequality 
\begin{equation}\label{tildePP}
\tilde{\Phi}(x,\varphi(x,t)t) \leq \Phi(x,2t), \quad x \in \Omega \quad \mbox{and} \quad t \geq 0.
\end{equation} 

The condition ($\varphi_3$) is very important, because following the ideas of \cite[Lemmas 2.1 and 2.5]{FN}, it is possible to prove the following: Setting the functions
$$
\xi_{0}(t)=\min\{t^{p},t^q\}, \quad \xi_{1}(t)=\max\{t^{p},t^q\}, \quad \xi_{3}(t)=\min\{t^{\frac{p}{p-1}},t^{\frac{q}{q-1}}\} \quad \mbox{and} \quad \xi_{4}(t)=\max\{t^{\frac{p}{p-1}},t^{\frac{q}{q-1}}\},
$$
we have
\begin{equation} \label{DES2}
\xi_0(s)\Phi(x,t) \leq \Phi(x,st) \leq \xi_1(s)\Phi(x,t) \quad \mbox{for} \quad s,t \geq 0,
\end{equation}
\begin{equation} \label{DES3}
\xi_0(|u|_\Phi) \leq \int_{\Omega}\Phi(x,|u|)\,dx \leq \xi_1(|u|_\Phi) \quad \mbox{for} \quad u \in L^{\Phi}(\Omega),
\end{equation}
\begin{equation} \label{DES4}
\xi_3(s)\tilde{\Phi}(x,t) \leq \tilde{\Phi}(x,st) \leq \xi_4(s)\tilde{\Phi}(x,t) \quad \mbox{for} \quad s,t \geq 0,
\end{equation}
and
\begin{equation} \label{DES5}
\xi_3(|u|_{\tilde{\Phi}}) \leq \int_{\Omega}\tilde{\Phi}(x,|u|)\,dx \leq \xi_4(|u|_{\tilde{\Phi}}) \quad \mbox{for} \quad u \in L^{\tilde{\Phi}}(\Omega).
\end{equation}

The Musielak-Sobolev space $W^{1,\Phi}(\Omega)$ can be defined by
\[
W^{1,\Phi}(\Omega)=\left\{  u\in
L^{\Phi}(\Omega)\left\vert \,\left\vert \nabla u\right\vert
\in L^{\Phi}(\Omega)\right. \right\}
\]
with the norm
\[
\left\Vert u\right\Vert _{1,\Phi}=\left\vert u\right\vert
_{\Phi}+\left\vert \nabla u\right\vert _{\Phi}\text{.}
\]

The conditions $(\varphi_1)-(\varphi_5)$ ensure that the spaces $L^{\Phi}(\Omega)$ and
$W^{1,\Phi}(\Omega)$ are reflexive and separable Banach spaces, for more details see \cite[Propositions 1.6 and 1.8]{Fan}. In what follows, $W_0^{1,\Phi}(\Omega)$ is defined as the closure of $C_0^{\infty}(\Omega)$ in $W_0^{1,\Phi}(\Omega)$ with respect to the above norm. Moreover, $\|u\|=|\nabla u|_{\Phi}$ is a norm in $W_0^{1,\Phi}(\Omega)$ and if $(\varphi_1)-(\varphi_5)$ holds, by \cite[Lemma 5.7]{Gossez}, $\|\,\,\|$ is equivalent to the norm $\|u\|_{1,\Phi}$ in $W_0^{1,\Phi}(\Omega)$.

As a consequence of (\ref{DES3}) we have the lemma below that will be used later on.

\begin{proposition} \label{modular1}
	\label{p1}The functional $\rho:W_0^{1,\Phi}(\Omega) \to \mathbb{R}$ defined by
	\begin{equation}
	\rho(u)=\int_{\Omega} \Phi(x,|\nabla u|)\,dx
	\text{,} \label{ps}
	\end{equation}
	has the following properties:	
	\begin{enumerate}
		\item[\emph{(i)}] If $\left\Vert u\right\Vert \geq1$, then $\left\Vert u\right\Vert
		^{p}\leq\rho(u)\leq\left\Vert u\right\Vert ^{q}$.
		
		\item[\emph{(ii)}] If $\| u\| \leq1$, then $\left\Vert u\right\Vert
		^{q}\leq\rho(u)\leq\left\Vert u\right\Vert ^{p}$.
	\end{enumerate}
	In particular, $\rho(u)=1$ if and only if $\left\Vert u \right\Vert =1$ and if $ (u_n) \subset W_0^{ 1,\Phi}( \Omega ) $, then $\left\Vert u_{n}\right\Vert 	\rightarrow0$ if and only if $\rho( u_{n}) \rightarrow0$.
\end{proposition}

\begin{remark}
	\label{r1}For the functional
	$\xi:L^{\Phi}(\Omega)\rightarrow \mathbb{R}$ given by
	\[
	\xi(u)=\int_{\Omega}\Phi(x,|u|)\,dx\text{,}
	\]
	the conclusion of Proposition \ref{p1} also holds, for example, if $ (u_n) \subset L^{\Phi}( \Omega) $, then  $\left\vert u_{n}\right\vert
	_{\Phi}\rightarrow0$ if and only if $\xi(u_{n})\rightarrow0$.
\end{remark}

From the definition of $W^{1,\Phi}(\Omega)$ and properties of $\Phi$, we have the  continuous embedding
$$
W^{1,\Phi}(\Omega) \hookrightarrow W^{1,q}(({\Omega}_q)_{\omega})
$$
for all $\omega \in (0,\delta)$ and the compact embedding
$$
W^{1,q}(({\Omega}_q)_\delta) \hookrightarrow C(\overline{(\Omega_q)_{\omega}}),
$$
because $q>N$, from where it follows that
\begin{equation} \label{EMB1}
W^{1,\Phi}(\Omega) \hookrightarrow C(\overline{(\Omega_q)_{\omega}}),
\end{equation}
is compact, which is crucial in our approach.

Next we would like to state our last result found in \cite[Theorem 2.2]{Fan}, which says the operator $-\Delta_{\Phi}: W_0^{1,\Phi}(\Omega) \to (W_0^{1,\Phi}(\Omega))^*$  belongs to the Class $(S_+)$.

\begin{lemma} \label{S+} Assume the conditions $(\varphi_1)-(\varphi_8)$. If $u_n \rightharpoonup u$ in $W_0^{1,\Phi}(\Omega)$ and
$$
	\lim_{n \to +\infty}\int_{\Omega}\langle \varphi(x,|\nabla u_n|)\nabla u_n, \nabla u_n- \nabla u\rangle\,dx=0,
$$
then $u_n \to u$ in $W_0^{1,\Phi}(\Omega)$.	
\end{lemma}

Before concluding this section, we will show a version of Poincar\'e's inequality, which is a key point in the proof of Theorem \ref{T2}.

\begin{lemma} \label{poincare} Assume $(\varphi_1)-(\varphi_5)$ and $(\varphi_{11})$. Then, there is $\Upsilon>0$ independent of $r\geq 1$ such that
$$
\int_{\Omega_r}\Phi(x,|u|)\,dx \leq \Upsilon \int_{\Omega_r}\Phi(x,|\nabla u|)\,dx, \quad \forall u \in W_0^{1,\Phi}(\Omega_r).
$$ 
\end{lemma}
\begin{proof} Fix $p>1$ and $v \in C_{0}^{\infty}(\Omega_r)$. Arguing as in \cite[Lemma 3.1]{AlvesFreitas}, we get 
$$
\int_{\Omega_r}|v|^{p}\,dx \leq \left(\frac{r+1}{r}\right)^{N-1} \int_{\Omega_r}|\nabla v|^{p}\,dx.  
$$
Now, taking the limit when $p \to 1$ and using the fact that $r \geq 1$, we derive that
$$
\int_{\Omega_r}|v|\,dx \leq 2^{N-1} \int_{\Omega_r}|\nabla v|\,dx, \quad \forall  v \in C_{0}^{\infty}(\Omega_r). 
$$	
Since $C_0^{\infty}(\mathbb{R}^N)$ is dense in $W_0^{1,1}(\Omega_r)$, it follows that  
$$
\int_{\Omega_r}|w|\,dx \leq 2^{N-1} \int_{\Omega_r}|\nabla w|\,dx, \quad \forall  w \in W_0^{1,1}(\Omega_r). 
$$
Now, for each $u \in W_0^{1,\Phi}(\Omega_r)$, we know that $w=\Phi(x,u) \in W_0^{1,1}(\Omega_r)$ and so, 
$$
\int_{\Omega_r}\Phi(x,|u|)\,dx \leq 2^{N-1} \int_{\Omega_r}|\nabla \Phi(x,u)|\,dx.
$$
Since $|\nabla \Phi(x,u)|\leq |\partial_s \Phi(x,|u|)|+\varphi(|u|)|u||\nabla u|$, we obtain
$$
\int_{\Omega_r}\Phi(x,|u|)\,dx \leq 2^{N-1} \int_{\Omega_r}\left(|\partial_s \Phi(x,|u|)|+\varphi(x,|u|)|u||\nabla u|\right)\,dx, \quad \forall  u \in W_0^{1,\Phi}(\Omega_r).
$$
Given $\epsilon>0$, by $\Delta_2$ condition, $(\varphi_{11})$, (\ref{YIne}) and (\ref{tildePP}), there is $C_\epsilon>0$ such that 
$$
\int_{\Omega_r}\Phi(x,|u|)\,dx \leq 2^{N-1} \left[\epsilon \int_{\Omega_r}\Phi(x,|u|)\,dx+\kappa \int_{\Omega_r}\Phi(x,|u|)\,dx+C_\epsilon \int_{\Omega_r}\Phi(x,|\nabla u|)\,dx\;\; \right],
$$
for all  $u \in W_0^{1,\Phi}(\Omega_r)$. Thus, for $\epsilon=\frac{1}{2^{N+1}}$ and recalling that $\kappa < \frac{1}{2^{N+1}}$, there is $\Upsilon>0$ independent of $r\geq1$ such that 
$$
\int_{\Omega_r}\Phi(x,|u|)\,dx \leq \Upsilon \int_{\Omega_r}\Phi(x,|\nabla u|)\,dx\;\; , \quad \forall  u \in W_0^{1,\Phi}(\Omega_r).
$$ 
\end{proof}

\section{Some technical results}

The main goal of this section is to recall and prove some technical results that are crucial in the proof of our main result. Since we are going to work with double criticality, which involves the exponential critical growth and the critical growth $p^*$, the next two results are crucial in our approach. The first one is a Concentration Compactness Lemma due to Lions for $W^{1,p}(\Theta)$ explored in Medeiros \cite{Everaldo}, where $\Theta  \subset \mathbb{R}^N$ is a smooth bounded domain.
\begin{lemma} \label{concentration} Let $(u_n)$  be a sequence 	in $W^{1,p}(\Theta)$ with $1<p<N$ and $u_n \rightharpoonup u$ in $W^{1,p}(\Theta)$. If \\
\noindent $(i)$	\,\, $|\nabla u_n|^{p} \to \mu$ weakly-$^*$ in the sense of measure,\\
\noindent and \\
\noindent $(ii)$ \,\, $|u_n|^{p^*} \to \nu$ weakly-$^*$ in the sense of measure, \\
then for at most a countable index set $J$, we have
$$
\left\{
\begin{array}{l}
(a)\quad  \nu=|u|^{p^*}+\sum_{j \in J}\nu_j \delta_{x_j},\,\nu_j \geq 0.\\
(b)\quad  \mu \geq |\nabla u|^{p}+\sum_{j \in J}\mu_j \delta_{x_j},\,\mu_j \geq 0.\\
(c)\,\, \mbox{If} \,\, x_j \in \Theta, \,\, \mbox{then} \quad S_p\nu_j^{\frac{p}{p^*}} \leq \mu_j.\\
(d)\,\, \mbox{If} \,\, x_j \in \partial \Theta, \,\, \mbox{then} \quad \frac{S_p}{2^{p/N}}\nu_j^{\frac{p}{p^*}} \leq \mu_j,
\end{array}
\right.
$$	
where $p^*=\frac{Np}{N-p}$ and $S_p$ denotes the best constant of the embedding $D^{1,p}(\mathbb{R}^N) \hookrightarrow L^{p^*}(\mathbb{R}^N)$ given by
\begin{equation} \label{Sp*}
S_p=
\inf_{ {\footnotesize{
			\begin{array}{l}
			u \in D^{1,p}(\mathbb{R}^N) \\
			u \not=0
			\end{array}}}}
\frac{ \int_{\mathbb{R}^N}|\nabla u|^{p}\,dx}{\left(\int_{\mathbb{R}^N}|u|^{p^*}\,dx\right)^{\frac{p}{p^*}}}.
\end{equation}
\end{lemma}

The proof of the above lemma follows by combining the arguments explored in Struwe \cite[Chapter I, Section 4]{Struwe} and the following Cherrier's inequality \cite{C} below.

\begin{lemma} \label{Ch} Let $\Theta \subset \mathbb{R}^N$ be a smooth bounded domain and $p \in (1,N)$. Then for each $\tau>0$, there is $M_{\tau}>0$ such that
	$$
	\left[\frac{S_p}{2^{\frac{p}{N}}}-\tau\right]\|u\|^{p}_{L^{p^*}(\Theta)} \leq  \|\nabla u\|^{p}_{L^{p}(\Theta)}+M_{\tau}\|u\|^{p}_{L^{p}(\Theta)}, \quad \forall\,u \in W^{1,p}(\Theta).
	$$
\end{lemma}

The second  result that we would like to point out is a version of Trundiger-Moser inequality in $W^{1,N}(\Theta)$ due to Cianchi \cite[Theorem 1.1]{Cianchi}.
\begin{lemma} \label{TM} Let $\Theta \subset \mathbb{R}^N$ be a smooth bounded domain for $N \geq 2$ and $ u \in W^{1,N}(\Theta)$. Then, there is a constant $C(\Theta)>0$ such that
\begin{equation} \label{TM1}
\int_{\Theta}e^{\alpha_N\left(\frac{|u-u_{\Theta}|}{\|\nabla u\|_{L^{N}(\Theta)}}\right)^{N'}}\,dx \leq C(\Theta),
\end{equation}
where $u_{\Theta}=\frac{1}{|\Theta|}\int_{\Theta}u\,dx$ is the mean value of $u$ in $\Theta$, $\alpha_N=N\left(\frac{w_N}{2}\right)^{\frac{1}{N}}$ and $w_N$ is the volume of sphere $S^{N-1}$. The integral on the left-hand of \eqref{TM1} is finite for each $u \in W^{1,N}(\Theta)$ even if $\alpha_N$ is replaced by any other small positive number, but no inequality of type   \eqref{TM1} can hold with a large constant in the place of $\alpha_N$.
\end{lemma}

From Lemma \ref{TM}, for each $u \in W^{1,N}(\Theta)$, we have
\begin{equation} \label{E0}
e^{t|u|^{N'}} \in L^{1}(\Theta), \quad \forall\,t \geq 0.
\end{equation}

For the reader interested in Trudinger-Moser inequality for functions in $W^{1,N}(\Theta)$, we would like to cite the papers due to Adimurthi and Yadava \cite{AY},  Kaur and Sreenadh \cite{KS} and their references.

As a consequence of Lemma \ref{TM}, we have the following two results whose proof can be found in \cite{AGR}.

\begin{lemma} \label{TM2} Given $t >1$ and $\alpha>0$, there is $r \in (0,1)$ and $C=C(t,r,N)>0$ such that
\begin{equation} \label{TM3}
\sup\left\{\int_{\Theta}e^{t\alpha|u|^{N'}}\,dx\,:\, u \in W^{1,N}(\Theta), \,\,  \|\nabla u\|_{L^{N}(\Theta)}\leq r \quad \mbox{and} \quad \|u\|_{L^{1}(\Theta)} \leq r \right\}\leq C.
\end{equation}	
\end{lemma}

\begin{lemma} \label{TM4} Let $\alpha>0$ and $(u_n) \subset W^{1,N}(\Theta)$ be a sequence satisfying $\|\nabla u_n\|^{N'}_{L^{N}(\Theta)}\leq \frac{\tau}{2^{N'}}\frac{\alpha_N}{\alpha}$ and $\|u_n\|_{L^{1}(\Theta)} \leq M$ for some $\tau \in (0,1)$ and $M>0$. Then, there is $t >1$ with $t \approx 1$ such that
\begin{equation} \label{TM5}
\sup_{n \in\mathbb{N}}\int_{\Theta}e^{t\alpha|u_n|^{N'}}\,dx<+\infty.
\end{equation}
Hence, the sequence $f_n(x)=e^{\alpha|u_n(x)|^{N'}}$ is bounded in $L^{t}(\Theta)$.	
\end{lemma}

As a consequence of Lemma \ref{TM4}, we have the corollary below.

\begin{corollary} \label{TM6} Let $(u_n) \subset W^{1,N}(\Theta)$ be a sequence as in Lemma  \ref{TM4}. If $u_n(x) \to u(x)$ a.e. in $\Theta$, then $f_n \rightharpoonup f$ in $L^{t}(\Theta)$ where $f(x)=e^{\alpha|u(x)|^{N'}}$, that is,
$$
\int_{\Theta}f_n \varphi \,dx \to \int_{\Theta}f \varphi \,dx, \quad \forall\,\varphi \in L^{t'}(\Theta),
$$		
where $\frac{1}{t}+\frac{1}{t'}=1$.	
\end{corollary}

Our next result will help us to conclude that the energy functional associated with the problem $(P)$ is $C^{1}(W_0^{1,\Phi}(\Omega),\mathbb{R})$. Since it follows as in Bezerra do \'O, Medeiros and Severo \cite[Proposition 1]{JEU}, we will omit its proof.

\begin{lemma}  \label{convegenceW} Let $(u_n) \subset W^{1,N}(\Theta)$ be a sequence such that $u_n \to u$ in $ W^{1,N}(\Theta)$ for some $u \in W^{1,N}(\Theta)$. Then, for some subsequence, still denoted by itself, there is $v \in W^{1,N}(\Theta)$ such that: \\
\noindent $(i)$\, $u_n(x) \to u(x)$ a.e. in $\Theta$. \\
\noindent $(ii)$\, $|u_n(x)| \leq v(x)\,\,$a.e. in $\Theta$ for all $n \in \mathbb{N}$.
\end{lemma}

\section{Preliminaries for the proof of Theorem \ref{T1}}
To prove Theorem \ref{T1}, we use the following result, whose proof follows similar arguments as in Ambrosetti and Rabinowitz \cite{AmRb}. Let $X$ be a Banach space, $K\subset X$ be compact and $\gamma(Y)$ be the genus of $Y\subset\Sigma,$ where 
$$
\Sigma:=\{Y\in X\setminus\{0\}:Y\text{ is close in $X$ and symmetric with respect to the origin}\}.
$$
\begin{theorem}\label{AmRbthm}
Suppose $I\in C^1(X,\mathbb{R})$ satisfies:
\begin{enumerate}
\item[(a)] $I(0)=0$, $I(u)=I(-u)$ for all $u\in X.$
\item[(b)] there exists $\alpha,\rho>0$ such that 
$$
I(u)\geq\alpha\,\,\forall\,\,||u||=\rho.
$$
\item[(c)] for every $\hat{X}\subset X$ such that $\text{dim}\,\hat{X}<\infty$, there exists $R=R(\hat{X})>0$, such that $$
I(u)\leq 0,\text{ for every }u\in\hat{X}\setminus B_{R}(0);
$$
\item[(d)] there exists $M>0$, such that $I$ satisfies $(PS)_c$ condition, for any $0<c<M.$ 
\end{enumerate}
For each $m\in\mathbb{N}$, fix a finite dimensional subspace $X_m$  of $X$ and consider $R_m=R(X_m)>0$ given by condition $(c).$ Now, define
	\begin{equation}\label{compact}
		D_m=B_{R_m}\cap X_m,
	\end{equation}
	$$
	G_m:=\{h\in C(D_m,X):h\text{ is odd and }h(u)=u, \;\; \forall\,u\in\partial B_{R_m}\cap X_m\},
	$$
	\begin{equation}\label{gamma}
		\Gamma_j=\{h\Big(\overline{D_m\setminus Y}\Big);\,h\in G_m,\,m\geq j,\,Y\in\Sigma,\,\gamma(Y)\leq m-j\},
	\end{equation}
and 
$$
c_m:=\inf_{K\in\Gamma_m}\max_{u\in K}\,I(u).
$$
Then $0<\alpha\leq c_m\leq c_{m+1},$ and if $c_m<M,$ the levels $c_j$ for $j \in\{1,2,...,m\}$ are critical values of $I$. Moreover, if $c_1=c_2=\cdots=c_r=c<M,$ then $\gamma(K_0)>r.$
\end{theorem}

\subsection{Functional setting}
In what follows, we consider the associated energy functional $I:W_0^{1,\Phi}(\Omega) \to \mathbb{R}$ given by
$$
I(u)=\int_{\Omega}\Phi(x,|\nabla u|)\,dx-\int_{\Omega}F(x,u)\,dx,
$$
where $F(x,t)=\int_{0}^{t}f(x,s)\, ds, \, t \in \mathbb{R}$ and $f$ is either of the form $(f_1)$ or $(f_2)$. Here, we would like to mention that, for the rest of the article, whenever we deal with $(f_1)$, we assume $\Omega$ to be a smooth bounded domain in $\mathbb{R}^N$, $N\geq 2$ along with the hypothesis $(g_0)-(g_3)$ and  $(\varphi_1)-(\varphi_8)$ as in Theorem \ref{T1}. For $(f_2)$, we consider $\Omega=\Omega_r$, $N\geq 2,$ $N\neq 3$ along with the hypothesis $(g_1),(g_2),(g_4), (\eta)$ and $(\Omega_N)$ and $(\varphi_1)-(\varphi_{11})$ as in Theorem \ref{T2}.

\begin{lemma} \label{L1} Assume that $f$ is of form $(f_1)$ or $(f_2)$. Then, the functional $I$ belongs to $C^{1}(W_0^{1,\Phi}(\Omega),\mathbb{R})$ and
$$
I'(u)v=\int_{\Omega}\varphi(x,|\nabla u|)\nabla u\nabla v\,dx-\int_{\Omega}f(x,u)v\,dx, \quad \forall u,v \in W_0^{1,\Phi}(\Omega).
$$	
\end{lemma}
\begin{proof} See proof in \cite[Lemma 3.8]{AGR}
\end{proof}	

Next, our goal is to prove that $I$ satisfies the geometric conditions of Theorem \ref{AmRbthm} and the well known $(PS)$ condition.

\begin{lemma} \label{LMP} Assume that $f$ is of the form $(f_1)$. Then, \\
\noindent $i)$\, There are $r, \rho>0$ such that
$$
I(u) \geq \rho, \quad \mbox{for} \quad \|u\|=r.
$$	
\noindent 
$ii)$\, For every $\hat{X}\subset W_{0}^{1,\Phi}(\Omega)$ with $\dim\,\hat{X}<\infty,$ there exists $R=R(\hat{X})>0$, such that 
$$
I(u)\leq 0,\text{ for all }u\in\hat{X}\setminus B_R(0).
$$
\end{lemma}
\begin{proof} \noindent The proof of $i)$ can be done as in \cite[Lemma 3.9]{AGR}.   

\noindent $ii)$ Suppose for each $n\in\mathbb{N},$ there exists $u_n\in\hat{X}\setminus B_n(0)$ such that 
\begin{equation}\label{Ipositive}
I(u_n)>0.
\end{equation}
From $(g_2)$, it follows that
\begin{equation}\label{AR1}
0<\chi F(x,t)\leq f(x,t)t,\,\forall\,(x,t)\in\Omega\times(\mathbb{R}\setminus\{0\}),
\end{equation}
where $\chi=\min\{\theta,\beta,\zeta\}>q.$ Therefore, $f$ satisfies the Ambrosetti-Rabinowitz condition. This gives the existence of positive constants $C,D$ such that 
\begin{equation}\label{Fprop}
F(x,t)\geq C|t|^{\chi}-D, \quad \forall t \in \mathbb{R} \quad \mbox{and} \quad \forall x \in \Omega. 
\end{equation}
Using Proposition \ref{modular1}-(ii) along with \eqref{Fprop}, we obtain
\begin{equation}\label{Inegative}
I(u_n)\leq ||u_n||^q-C\int_{\Omega}|u_n|^{\chi}\,dx+D|\Omega|.
\end{equation}
Since $\dim\,\hat{X}<\infty$ and $\chi>q$ letting $n\to\infty$ in \eqref{Inegative}, we arrive at a contradiction to our assumption \eqref{Ipositive}. Hence $ii)$ follows.
\end{proof}

\begin{lemma} \label{boundedness} Assume that $f$ is of the type $(f_1)$ or $(f_2)$. Then, every $(PS)$ sequence $(u_n)$ of the functional $I$ is bounded in $W_0^{1,\Phi}(\Omega)$.
\end{lemma}
\begin{proof} 
Let $d>0$ and $(u_n)$ be a $(PS)_d$ sequence for $I$. Then, there are constants $C_1,C_2>0$ such that
\begin{equation} \label{AR2}
I(u_n)-\frac{1}{\chi}I'(u_n)u_n \leq C_1+C_2\|u_n\|, \quad \forall\,n \in \mathbb{N}.
\end{equation}
If $f$ is of the type $(f_1)$ or $(f_2)$ it is easy to check that \eqref{AR1} holds. Hence, by $(\varphi_3)$ and the definition of $I$, 
\begin{equation*}
\begin{split}
I(u_n)-\frac{1}{\chi}I'(u_n)u_n&\geq \int_{\Omega}\Phi(x,|\nabla u_n|)\,dx-\frac{1}{\chi}\int_{\Omega}\varphi(x,|\nabla u_n|)|\nabla u_n|^2\,dx\\
&\geq \left(1-\frac{q}{\chi}\right)\int_{\Omega}\Phi(x,|\nabla u_n|)\,dx.
\end{split}
\end{equation*}
Therefore,
$$
\left(1-\frac{q}{\chi}\right)\int_{\Omega}\Phi(x,|\nabla u_n|)\,dx \leq C_1+C_2\|u_n\|, \quad \forall \,n \in \mathbb{N}.
$$
If there is $(u_{n_j}) \subset (u_n)$ such that $\|u_{n_j}\| \geq 1$, then Proposition \ref{modular1}-(i) leads to
$$
\left(1-\frac{q}{\chi}\right)\|u_{n_j}\|^{p}\leq C_1+C_2\|u_{n_j}\|, \quad \forall\,j \in \mathbb{N},
$$
from where it follows the boundedness of $(u_{n_j})$. This implies the boundedness of  $(u_n)$. 
\end{proof}

\begin{corollary} \label{C1} Assume that $f$ is of the type $(f_1)$ or $(f_2)$ and let $(u_n)$ be a $(PS)_d$ sequence of $I$ with $d \in (0,M)$,
	$$
	M=\left(1-\frac{q}{\chi}\right)\min\left\{ \frac{1}{N}\left(\frac{\alpha_N}{2^{N'}\alpha}\right)^{N-1}, \frac{1}{p} {S_p^{\frac{N}{p}}}\right\},
	$$
    where $\chi=\min\{\theta,\beta,\zeta\}. $ Then, 
	$$
	\limsup_{n\to +\infty}\|\nabla u_n\|_{L^N(\Omega_N)}^{N'}<\frac{\alpha_N}{2^{N'}\alpha}.
	$$ 
	Hence, without loss of generality, we can assume that there is $\tau \in (0,1)$  such that 
	$$
	\|\nabla u_n\|_{L^{N}(\Omega_N)}^{N'} \leq \frac{\tau \alpha_N}{2^{N'} \alpha}, \quad \forall n \in \mathbb{N}.
	$$	
\end{corollary}
\begin{proof} First of all, we must recall that
	$$
	I(u_n)-\frac{1}{\chi}I'(u_n)u_n=d+o_n(1)\|u_n\|+o_n(1).
	$$
	Therefore, by $(\varphi_6)$, 
	\begin{equation*}
	\begin{split}
	d+o_n(1)\|u_n\|+o_n(1)&\geq \int_{\Omega}\left((\Phi(x,|\nabla u_n|)-\frac{1}{\chi}\varphi(x,|\nabla u_n|)|\nabla u_n|^{2}\right)\,dx\\ 
	&\geq \frac{1}{N}\left(1-\frac{q}{\chi}\right)\int_{\Omega_N}|\nabla u_n|^{N}\,dx.
	\end{split}
	\end{equation*}
	Hence, 
	$$
	\limsup_{n \to +\infty}\frac{1}{N}\left(1-\frac{q}{\chi}\right)\int_{\Omega_N}|\nabla u_n|^{N}\,dx \leq d<   \min\left(1-\frac{q}{\chi}\right)\left\{\frac{1}{N}\left(\frac{\alpha_N}{2^{N'}\alpha}\right)^{N-1},\frac{1}{p}S_p^{\frac{N}{p}}\right\}
	$$
	leading to
	$$
	\limsup_{n \to +\infty}\int_{\Omega_N}|\nabla u_n|^{N}\,dx < \left(\frac{\alpha_N}{2^{N'}\alpha}\right)^{N-1},
	$$
	which proves the lemma.
	
\end{proof}

\begin{lemma} \label{PS} Assume that $f$ is of the type $(f_1)$ or $(f_2)$. Then, the functional $I$ verifies the $(PS)_d$ condition for $d \in (0,M)$, where $M$ was given in Corollary \ref{C1}.
\end{lemma}
\begin{proof} The proof this lemma follows as in \cite[Lemma 3.13]{AGR}, however for the reader's convenience, we will write the proof for $(f_1)$, since for $(f_2)$, it follows with similar arguments. Let $(u_n)$ be a $(PS)_d$ sequence for $I$. Then, by Lemma \ref{boundedness}, $(u_n)$ is bounded in $W_0^{1,\Phi}(\Omega)$. Since $W_0^{1,\Phi}(\Omega)$ is reflexive, we assume that for some subsequence, still denoted by itself, there is $u \in W_0^{1,\Phi}(\Omega)$ such that
	$$
	u_n \rightharpoonup u \quad \mbox{in} \quad W_0^{1,\Phi}(\Omega),
	$$
	and
	\begin{equation*} \label{CVG1}
		u_n(x) \to u(x) \quad \mbox{a.e. in} \quad \Omega.
	\end{equation*}
	
	Let us set
	$$
	P_n=\int_{\Omega}\langle \varphi(x,|\nabla u_n|)\nabla u_n, \nabla u_n- \nabla u\rangle\,dx,
	$$	
	that is,
	$$
	P_n=I'(u_n)u_n+\int_{\Omega}f(x,u_n)u_n\,dx-I'(u_n)u-\int_{\Omega}f(x,u_n)u\,dx.
	$$	
	Consequently
	$$
	P_n=\int_{\Omega}f(x,u_n)u_n\,dx-\int_{\Omega}f(x,u_n)u\,dx+o_n(1).
	$$
	From the definition of $f$ together with embedding (\ref{EMB1}),
	$$
	\lim_{n \to +\infty}\int_{\Omega}\tilde{\eta}_q(x)g(x,u_n)u_n\,dx=\lim_{n \to +\infty}\int_{\Omega}\tilde{\eta}_q(x)g(x,u_n)u\,dx=\int_{\Omega}\tilde{\eta}_q(x)g(x,u)u\,dx,
	$$
	$$
	\lim_{n \to +\infty}\int_{\Omega}{\eta}_p(x)|u_n|^\zeta\,dx=\lim_{n \to +\infty}\int_{\Omega}{\eta}_p(x)|u_n|^{\zeta-2}u_nu\,dx=\int_{\Omega}{\eta}_p(x)|u|^\zeta\,dx,
	$$
	$$
	\lim_{n \to +\infty}\int_{\Omega \setminus \Omega_N}{\eta}_N(x)|u_n|^{\beta}e^{\alpha|u_n|^{N'}}\,dx=\int_{\Omega \setminus \Omega_N}{\eta}_N(x)|u|^{\beta}e^{\alpha|u|^{N'}}\,dx,
	$$
	$$
	\lim_{n \to +\infty}\int_{\Omega \setminus \Omega_N}{\eta}_N(x)|u_n|^{\beta-2}u_{n}ue^{\alpha|u_n|^{N'}}dx=\int_{\Omega \setminus \Omega_N}{\eta}_N(x)|u|^{\beta}e^{\alpha|u|^{N'}}\,dx,
	$$
	$$
	\lim_{n \to +\infty}\int_{\Omega \setminus \Omega_p}{\eta}_p(x)|u_n|^{p^*}\,dx=\int_{\Omega \setminus \Omega_p}{\eta}_p(x)|u|^{p^*}\,dx,
	$$
	and
	$$
	\lim_{n \to +\infty}\int_{\Omega \setminus \Omega_p}{\eta}_p(x)|u_n|^{p^*-2}u_nu\,dx=\int_{\Omega \setminus \Omega_p}{\eta}_p(x)|u|^{p^*}\,dx.
	$$
	Consequently
	$$
	\begin{array}{l}
		P_n=\displaystyle \lambda \int_{\Omega_N}|u_n|^{\beta}e^{\alpha|u_n|^{N'}}\,dx-\lambda \int_{\Omega_N }|u_n|^{\beta-2}u_nue^{\alpha|u_n|^{N'}}\,dx
		+\displaystyle \int_{\Omega_p}|u_n|^{p^*}\,dx \\
		\mbox{} \\
		\hspace{2 cm} -\displaystyle \int_{\Omega_p}|u_n|^{p^*-2}u_nu\,dx\,dx+o_n(1).
	\end{array}
	$$
	By Corollary \ref{C1}, the sequence $(u_n)$ satisfies
	$$
	\|\nabla u_n\|_{L^{N}(\Omega_N)}^{N'} \leq \frac{\tau \alpha_N}{2^{N'} \alpha}, \quad \forall n \in \mathbb{N},
	$$
	for some  $\tau \in (0,1)$.	Employing Corollary \ref{TM6}, there is $t>1$ and $t \approx 1$ such that the  sequence $h_n(x)=e^{\alpha|u_n(x)|^{N'}}$ is weakly convergent to $h(x)=e^{\alpha|u(x)|^{N'}}$ in $L^{t}(\Omega_N)$ , that is,
	\begin{equation} \label{LIMITE1}
		\int_{\Omega_N}h_n \varphi \,dx \to \int_{\Omega_N}h \varphi \,dx, \quad \forall \varphi \in L^{t'}(\Omega_N).
	\end{equation} 		
	As
	$$
	|u_n|^{\beta} \to |u|^{\beta} \quad \mbox{in} \quad L^{t'}(\Omega_N),
	$$	
	it follows that
	$$
	\int_{\Omega_N}h_n |u_n|^{\beta} \,dx \to \int_{\Omega_N}h |u|^\beta \,dx,
	$$	
	that is,
	$$
	\int_{\Omega_N}|u_n|^{\beta}e^{\alpha|u_n|^{N'}} \,dx \to \int_{\Omega_N}|u|^{\beta}e^{\alpha|u|^{N'}} \,dx.
	$$
	Now, using the fact that
	$$
	|u_n|^{\beta-2}u_nu \to |u|^{\beta} \quad \mbox{in} \quad L^{t'}(\Omega_N),
	$$
	we also derive that
	$$
	\int_{\Omega_N}|u_n|^{\beta-2}u_nue^{\alpha|u_n(x)|^{N'}} \,dx \to \int_{\Omega_N}|u|^{\beta-2}uue^{\alpha|u(x)|^{N'}}\,dx.
	$$
	The above analysis ensures that
	$$
	\lim_{n \to +\infty}\int_{\Omega_N}|u_n|^{\beta}e^{\alpha|u_n(x)|^{N'}}\,dx=\lim_{n \to +\infty}\int_{\Omega_N}|u_n|^{\beta-2}u_nue^{\alpha|u_n(x)|^{N'}}\,dx=\int_{\Omega_N}|u|^{\beta}e^{\alpha|u|^{N'}}\,dx,
	$$
	and then,
	$$
	P_n=\int_{\Omega_p}|u_n|^{p^*}\,dx-\int_{\Omega_p}|u_n|^{p^*-2}u_nu\,dx+o_n(1).
	$$
	By \cite[Lemma 4.8]{Kavian},
	$$
	\lim_{n \to +\infty}\int_{\Omega_p}|u_n|^{p^*-2}u_nu\,dx=\int_{\Omega_p}|u|^{p^*}\,dx,
	$$
	then
	$$
	P_n=\int_{\Omega_p}|u_n|^{p^*}\,dx-\int_{\Omega_p}|u|^{p^*}\,dx+o_n(1).
	$$
	Now, we are going to use the Concentration Compactness Lemma \ref{concentration} to the sequence $(u_n) \subset W^{1,p}(\Omega_p)$. From $(\varphi_7)$, for each open ball  $B \subset (\Omega_q)_{\delta} $ we have that the embedding $W^{1,\Phi}(\Omega) \hookrightarrow C(\overline{B})$ is compact, then as $(u_n)$ is a bounded $(PS)$ for $I$, it is possible to prove that for some subsequence there holds  
	$$
	\int_{B}\langle \varphi(x,|\nabla u_n|)\nabla u_n, \nabla u_n- \nabla u\rangle\,dx \to 0.
	$$
	Since from $(\varphi_6)-(\varphi_8)$, the embedding $W^{1,\Phi}(B) \hookrightarrow L^{\Phi}(B)$ is compact, the last limit together with the $\Delta_2$-condition \eqref{delta2} implies that 
	$$
	u_n \to u \quad \mbox{in} \quad  W^{1,\Phi}(B).
	$$
	Now, recalling that the embedding $W^{1,\Phi}(B) \hookrightarrow W^{1,p}(B)$ is continuous, we derive that  
	$$
	u_n \to u \quad \mbox{in} \quad  W^{1,p}(B),
	$$
	from where it follows that $x_i \in \overline{\Omega_p} \setminus  (\Omega_q)_{\delta}$ for all $i \in J$. Now, our goal is proving that $J$ must be a finite set. Have this in mind, we will consider $J=J_1 \cup J_2$ where 
	$$
	J_1=\{i \in J\,:\,x_i \in \overline{\Omega_p} \setminus  \overline{(\Omega_q)_{\delta}} \}
	$$
	and
	$$
	J_2=\{i \in J\,:\,x_i \in \partial (\Omega_q)_{\delta} \cap \Omega_p \}.
	$$
	
	If $i \in J_1$, the condition $(\varphi_7)$ says that $c_3t^{p-2} \geq \varphi(x,t) \geq t^{p-2}$ for $x \in \overline{\Omega_p} \setminus \overline{(\Omega_q)_{\delta}}$. This fact permits us to repeat the same arguments explored in \cite[Lemma 2.3]{GP} to conclude that $J_1$ is finite. Now, if $i \in J_2$, the situation is more subtle and we must be careful. In what follows let us consider $\tilde{\psi} \in C_{0}^{\infty}(\mathbb{R}^N)$ such that
	$$
	\tilde{\psi} \equiv 1 \;\; \mbox{on} \;\; B(0,1) \;\; \mbox{and} \;\; \tilde{\psi} \equiv 0 \;\; \mbox{on} \;\; B(0,2)^c.
	$$
	For each $\epsilon>0$, we set
	$$
	\psi(x)=\tilde{\psi}({(x-x_i)}/{\epsilon}), \quad \forall x \in \mathbb{R}^N.
	$$

	Since $(u_n)$ is a bounded sequence in $W^{1,\Phi}(\Omega)$, the sequence $(\psi u_n)$ is also bounded in $W^{1,\Phi}(\Omega)$ and so, $I'(u_n)\psi u_n=o_n(1)$. Hence,   
	$$
	\int_{\Omega}\varphi(x,|\nabla u_n|)\nabla u_n \nabla(\psi u_n)\,dx=\int_{\Omega}\tilde{\eta}_q(x)g(x,u_n)\psi u_n\,dx+\int_{\Omega}\eta_p(x)|u_n|^{p^*}\psi\,dx+o_n(1).
	$$
	Now, given $\xi>0$, the Young's inequality (\ref{YIne}) combined with (\ref{tildePP}) and $\Delta_2$-condition \eqref{delta2} gives  
	$$
	\int_{\Omega} |\varphi(x,|\nabla u_n|)|\nabla u_n||u_n||\nabla \psi|\,dx \leq \xi \int_{\Omega}\Phi(x,|\nabla u_n|)\,dx +C_{\xi}\int_{\Omega}\Phi(x,|\nabla \psi||u_n|)\,dx.
	$$
	for some $C_\xi>0$. Note that by $(\varphi_8)$,
	$$
	\int_{\Omega}\Phi(x,|\nabla \psi||u_n|)\,dx \leq C_1\left(\int_{B(x_i,2\epsilon)}|\nabla \psi|^p||u_n|^p\,dx+\int_{B(x_i,2\epsilon)}\tau_2(x)|\nabla \psi|^q||u_n|^q\,dx \right).
	$$
	By H\"older's inequality
	$$
	\limsup_{n \to +\infty}\int_{B(x_i,2\epsilon)} |u_n|^p|\nabla \psi|^p\,dx \leq C_2\left(\int_{B(x_i,2\epsilon)}|u|^{p^*}\,dx\right)^{\frac{N-p}{N}}
	$$
	from where it follows that
	\begin{equation} \label{INE1}
		\lim_{\epsilon \to 0}\left[\limsup_{n \to +\infty}\int_{B(x_i,2\epsilon)} |u_n|^p|\nabla \psi|^p\,dx\right] \leq\lim_{\epsilon\to 0} C_2\left(\int_{B(x_i,2\epsilon)}|u|^{p^*}\,dx\right)^{\frac{N-p}{N}}=0.
	\end{equation}
	
	Arguing as above, we also have
	$$
	\limsup_{n \to +\infty}\int_{\Omega}\tau_2(x)|u_n|^q|\nabla \psi|^q\,dx \leq \left(\int_{B(x_i,2\epsilon)}\left|\tau_2^{\frac{1}{q}}(x)\nabla \psi\right|^{\frac{qp^*}{p^*-q}}\,dx \right)^{\frac{p^*-q}{p^*}}\left(\int_{B(x_i,2\epsilon)}|u|^{p^*} \,dx\right)^{\frac{q}{p^*}}.
	$$
	By change of variable, 
	$$
	\begin{array}{l}
		\displaystyle\int_{B(x_i,2\epsilon)}\left|\tau_2^{\frac{1}{q}}(x)\nabla \psi\right|^{\frac{qp^*}{p^*-q}}\,dx=\left(\frac{1}{\epsilon}\right)^{\frac{qp^*}{p^*-q}}\int_{B(0,2)}\left|\tau_2^{\frac{1}{q}}(\epsilon x+x_i)\nabla \tilde{\psi}\right|^{\frac{qp^*}{p^*-q}}\,dx \\
		\mbox{} \\
		\hspace{4.5 cm}  \leq C_5\left(\frac{1}{\epsilon}\right)^{\frac{qp^*}{p^*-q}}\displaystyle \int_{B(0,2)}\left|\tau_2^{\frac{1}{q}}(\epsilon x+x_i)\right|^{\frac{qp^*}{p^*-q}}\,dx .
	\end{array}
	$$
	Since $x_i \in \partial (\Omega_q)_{\delta} \cap \Omega_p$, it follows that
	$$
	\tau_2(\epsilon x+x_i) \leq c_4\epsilon^{s}|x|^{s}
	$$
	and 
	$$
	\int_{B(x_i,2\epsilon)}\left|\tau_2^{\frac{1}{q}}(x)\nabla \psi\right|^{\frac{qp^*}{p^*-q}}\,dx \leq C_6\epsilon^{\frac{(s-q)p^*}{p^*-q}}.
	$$
	As $s>q$, it follows that 
	\begin{equation} \label{INE2}
		\lim_{\epsilon \to 0}\left[\limsup_{n \to +\infty}\int_{\Omega}\tau_2(x)|u_n|^q|\nabla \psi|^q\,dx \right]=0.
	\end{equation}
	Now, the boundedness of $(u_n)$ in $W^{1,\Phi}(\Omega)$ together with Proposition (\ref{modular1}), (\ref{INE1}) and (\ref{INE2}) ensures that
	$$
	\lim_{\epsilon \to 0}\left[\limsup_{n \to +\infty}\int_{\Omega}|\varphi(x,|\nabla u_n|)|\nabla u_n||u_n||\nabla \psi|\,dx \right] \leq \xi C,
	$$
	for some $C>0$. Since $\xi>0$ is arbitrary, we can deduce that
	$$
	\lim_{\epsilon \to 0}\left[\limsup_{n \to +\infty}\int_{\Omega}|\varphi(x,|\nabla u_n|)|\nabla u_n||u_n||\nabla \psi|\,dx \right]=0.
	$$
	The last limit  together with the fact that $\varphi(x,t)  \geq t^{p-2}$ for $x \in \Omega_p$ permit us to conclude as in \cite[Lemma 2.3]{GP}, that $J_2$ is also finite. Consequently, $J$ is a finite set.  However, in order to conclude the proof of the lemma, we need to show that $J$ is in fact an empty set. Seeking by a contradiction, assume that there is $i \in J$. In this case, the argument explored in \cite{GP} also says for us that
	$$
	\nu_i \geq {S_p^{\frac{N}{p}}}.
	$$
	Hence, by Lemma \ref{concentration}-(d),
	$$
	\mu_i \geq {S_p^{\frac{N}{p}}}.
	$$
	As $|\nabla u_n|^{p} \to \mu$ weakly-$^*$ in the sense of measure, we have
	$$
	\liminf_{n \to +\infty}\int_{\Omega_p}|\nabla u_n|^{p}\,dx \geq \mu_i
	$$
	and so,
	$$
	\liminf_{n \to +\infty}\int_{\Omega_p}|\nabla u_n|^{p}\,dx \geq {S_p^{\frac{N}{p}}}.
	$$
	
	Now, using once more the equality
	$$
	I(u_n)-\frac{1}{\chi}I'(u_n)u_n=d+o_n(1)\|u_n\|+o_n(1),
	$$
	we get
	$$
	d+o_n(1)\|u_n\|+o_n(1) \geq \frac{1}{p}\left(1-\frac{q}{\chi}\right)\int_{\Omega_p}|\nabla u_n|^{p}\,dx.
	$$
	Taking the limit of $n \to +\infty$, we find the inequality below
	$$
	d \geq \frac{1}{p}\left(1-\frac{q}{\chi}\right){S_p^{\frac{N}{p}}},
	$$
	which is a contradiction, showing that $J=\emptyset$. Thereby, by Lemma \ref{concentration}-$(a)$, $\nu=|u|^{p^*}$ and
	$$
	\int_{\Omega_p}|u_n|^{p^*}\,dx \to \int_{\Omega_p}|u|^{p^*}\,dx,
	$$
	implying that $P_n=o_n(1)$, that is,
	$$
	\lim_{n \to +\infty}\int_{\Omega}\langle \varphi(x,|\nabla u_n|)\nabla u_n, \nabla u_n- \nabla u\rangle\,dx=0.
	$$
	Now, it is enough to apply Lemma \ref{S+} to finish the proof.
\end{proof}

Our last lemma in this section is as follows:
\begin{lemma}\label{level}
Assume that $(f_1)$ holds. Then, for each $m\in \mathbb{N},$ there exists positive constants $\lambda_m, \mu_m$ and $\tau_m$ such that 
$$
c_m^{\lambda,\mu,\tau}:=\inf_{K\in \Gamma_m}\sup_{u\in K}\,I(u)<M,
$$
for all $\lambda \geq \lambda_m, \mu \ge \mu_m$ and $\tau \geq \tau_m$, where $M$ is given by Corollary \ref{C1}.
\end{lemma}
\begin{proof}
First, we claim that for some positive constant $C>0,$ we have 
\begin{equation}\label{claim}
\min_{u\in K,\,||u||=1}\Big\{\int_{\Omega_N}|u|^{\beta}\,dx+\int_{\Omega_q}|u|^{q_2}\,dx+\int_{\Omega_p}|u|^\zeta\,dx\Big\}\geq C,
\end{equation}
where $K\subset X_m$ is compact such that $\dim\,X_m<\infty.$

Indeed, if \eqref{claim} does not hold, then there exists a sequence $\{u_n\}\subset K$ with $||u_n||=1$ such that 
\begin{equation}\label{claim1}
\Big\{\int_{\Omega_N}|u_n|^{\beta}\,dx+\int_{\Omega_q}|u_n|^{q_2}\,dx+\int_{\Omega_p}|u_n|^\zeta\,dx\big\}\leq\frac{1}{n},\,\,\forall\,n\in\mathbb{N}.
\end{equation}
Since $\dim\,X_m<\infty,$ there exists a subsequence of $\{u_n\}$ still denoted by $\{u_n\}$ and $u\in K$ with $||u||=1$ such that $u_{n_j}\to u$ in $X_m.$ Then, letting $n\to\infty$ in \eqref{claim1}, we obtain
$$
\int_{\Omega_N}|u|^{\beta}\,dx+\int_{\Omega_q}|u|^{q_2}\,dx+\int_{\Omega_p}|u|^\zeta\,dx=0.
$$
Hence, we have $u=0$ a.e. in each of the sets $\Omega_N,\Omega_q$ and $\Omega_p.$
Now, since $\Omega=\Omega_N\cup\Omega_q\cup\Omega_p,$ we have $u=0$ a.e. in $\Omega.$ This contradicts the fact that $||u||=1.$ Hence the Claim \eqref{claim} follows.

Now we choose $K=\overline{D}_m$ where $D_m$ is given by \eqref{compact}. Since $h=I_d \in G_m$, the definition of $\Phi$ combined with (\ref{DES2}) and $(f_1)$ gives
$$
c_m^{\lambda,\mu,\tau}\leq\sup_{u\in K}\Big\{\|u\|^{p}+\|u\|^q-\lambda\int_{\Omega_N }|u|^{\beta}\,dx-\mu\int_{\Omega_q}|u|^{q_2}\,dx -\tau\int_{\Omega_p}|u|^{\zeta}\,dx\Big\},
$$
or equivalently
$$
c_m^{\lambda,\mu,\tau}\leq\sup_{u\in K}\Big\{\|u\|^{p}+\|u\|^q-\lambda\|u\|^{\beta}\int_{\Omega_N }\left|\frac{u}{\|u\|}\right|^{\beta}\,dx-\mu\|u\|^{q_2}\int_{\Omega_q}\left|\frac{u}{\|u\|}\right|^{q_2}\,dx -\tau\|u\|^\zeta\int_{\Omega_p}\left|\frac{u}{\|u\|}\right|^{\zeta}\,dx\Big\}.
$$
Now, when $||u||>1,$ we observe that
\begin{align*}
&\|u\|^{p}+\|u\|^q-\lambda\|u\|^{\beta}\int_{\Omega_N }\left|\frac{u}{\|u\|}\right|^{\beta}\,dx-\mu\|u\|^{q_2}\int_{\Omega_q}\left|\frac{u}{\|u\|}\right|^{q_2}\,dx -\tau\|u\|^r\int_{\Omega_p}\left|\frac{u}{\|u\|}\right|^{\zeta}\,dx\\
&\leq \|u\|^{p}+\|u\|^q-\|u\|^{l_1}\chi_1\Big(\int_{\Omega_N }\left|\frac{u}{\|u\|}\right|^{\beta}\,dx+\int_{\Omega_q}\left|\frac{u}{\|u\|}\right|^{q_2}\,dx +\int_{\Omega_p}\left|\frac{u}{\|u\|}\right|^{\zeta}\,dx\Big)\\
&\leq \|u\|^{p}+\|u\|^q-C\chi_1\|u\|^{l_1},
\end{align*}
where $l_1=\min\{\beta,q_2,\zeta\}$, $\chi_1=\min\{\lambda,\mu,\tau\}$ and the constant $C$ is given by \eqref{claim}. 

Moreover, when $\|u\|\leq 1,$ we get
\begin{align*}
&\|u\|^{p}+\|u\|^q-\lambda\|u\|^{\beta}\int_{\Omega_N }\left|\frac{u}{\|u\|}\right|^{\beta}\,dx-\mu\|u\|^{q_2}\int_{\Omega_q}\left|\frac{u}{\|u\|}\right|^{q_2}\,dx -\tau\|u\|^r\int_{\Omega_p}\left|\frac{u}{\|u\|}\right|^{\zeta}\,dx\\
&\leq \|u\|^{p}+\|u\|^q-\|u\|^{l_2}\chi_1\Big(\int_{\Omega_N }\left|\frac{u}{\|u\|}\right|^{\beta}\,dx+\int_{\Omega_q}\left|\frac{u}{\|u\|}\right|^{q_2}\,dx +\int_{\Omega_p}\left|\frac{u}{\|u\|}\right|^{\zeta}\,dx\Big)\\
&\leq \|u\|^{p}+\|u\|^q-C\chi_1\|u\|^{l_2},
\end{align*}
where $l_2=\max\{\beta,q_2,\zeta\}$, $\chi_1=\min\{\lambda,\mu,\tau\}$ and the constant $C$ is given by \eqref{claim}. Hence,
$$
c_m^{\lambda,\mu,\tau}\leq \|u\|^{p}+\|u\|^q-C\chi\|u\|^{l}, 
$$
where $l=l_1$ or $l_2.$

Let, $w(t)=t^p+t^{q}-C\chi t^l$. Then using the fact that $l>q>p,$ it can be easily seen that $w$ achieves its maximum at $\hat{t}=\hat{t}(\lambda,\beta,\tau)>0$ which goes to $0$ as the parameters $\lambda,\mu,\tau$ goes to infinity. Hence there exists $\lambda_m,\mu_m,\tau_m$ such that for all $\lambda\geq\lambda_m,\mu\geq\mu_m$ and $\tau\geq\tau_m,$ we have
$$
c_m^{\lambda,\mu,\tau}<M,
$$
where $M$ is given by Corollary \ref{C1}. Hence the Lemma follows.
\end{proof}

\section{Preliminaries for the proof of Theorem \ref{T2}}
\subsection{Functional setting}
In what follows $f$ is of the type $(f_2)$, $\Omega=\Omega_r$, see (\ref{omegar}), $N\geq 2$, $N\neq 3$ and the hypothesis $(g_1), (g_2), (g_4), (\varphi_1)-(\varphi_{11})$ will be assumed, unless otherwise mentioned. Let us denote by $O(N)$ the group of $N\times N$ orthogonal matrices. For any integer $1\leq k<\infty$, let us consider the finite rotational subgroup $O_k$ of $O(2)$ given by 
$$
O_k:=\left\{h\in O(2): \, h(x)=\left(x_1\cos\frac{2\pi l}{k}+x_2\sin\frac{2\pi l}{k},-x_1\sin\frac{2\pi l}{k}+x_2\cos\frac{2\pi l}{k}\right)\right\}
$$
where $x=(x_1,x_2)\in\mathbb{R}^2$ and $l\in \{0,\ldots,k-1\}$. We define the subgroups of $O(N)$ 
$$
H_k:=O_k\times O(N-2), \, 1\leq k<\infty\,\,\text{and}\,\, H_{\infty}:= O(2)\times O(N-2).
$$
Associated with the above subgroups, we set the subspaces
$$
W^{1,\Phi}_{0,H_k}(\Omega_r):=\left\{u\in W^{1,\Phi}_0(\Omega_r): u(x)=u(h^{-1} x),\, \text{for all}\, h\in H_k\right\},\, 1\leq k\leq\infty,
$$
endowed with the usual norm of $W^{1,\Phi}_0(\Omega_r)$, that is,
$$
||u||=||\nabla u||_{\Phi}+||u||_{\Phi}.
$$
Hereafter, we denote by $I:W^{1,\Phi}_{0,H_k}(\Omega_r)\to \mathbb{R}$ the functional given by 
$$
I(u)=\int_{\Omega_r}\Phi(x,|\nabla u|)dx-\int_{\Omega_r}F(x,u)dx.
$$
Throughout this section, $J_{k,r}$ denotes the following  real number 
$$
J_{k,r}=\inf_{u\in \mathcal{M}_{k,r}} I(u),
$$
where
$$
\mathcal{M}_{k,r}=\{u\in W^{1,\Phi}_{0,H_k}(\Omega_r)\setminus \{0\}, I'(u)u=0\}.
$$

\subsection{Properties of the levels $J_{k,r}$}
Our first result concerns the positivity of $J_{k,r}$.
\begin{lemma}\label{pos}
For any $1\leq k\leq\infty$ and $r>0$, we have $J_{k,r}>0$.
\end{lemma}
\begin{proof}
We prove the result in two steps.

\noindent \textbf{Step 1.} We claim that for every fixed $1\leq k\leq\infty$ and $r>0$, there exists a constant $\eta>0$ such that 
\begin{equation}\label{auxpos}
||u||>\eta,\text{ for all }u\in\mathcal{M}_{k,r}.
\end{equation}
Indeed, if \eqref{auxpos} does not hold, there exists a sequence  $(u_n)\in\mathcal{M}_{k,r}$ such that $||u_n||\to 0$ as $n\to\infty$. From $u_n\in\mathcal{M}_{k,r}$ we have $I'(u_n)u_n=0$. Hence
\begin{equation}\label{pos1}
\begin{split}
\int_{\Omega_r}\varphi(x,\nabla u_n)|\nabla u_n|^2 dx &=\int_{\Omega_r} f(|x|,u_n)u_n dx.
\end{split}
\end{equation}
Due to the fact $||u_n||\to 0$ as $n\to\infty$, without loss of generality we may assume that $||u_n||<1$ for all $n \in\mathbb{N}$. Hence from $(\varphi_3)$ and Proposition \ref{modular1},  we can estimate the left hand side of \eqref{pos1} as follows:
\begin{equation}\label{pos1lhs}
\begin{split}
p||u_n||^q\leq p\int_{\Omega_r}\Phi(x,u_n)\,dx\leq \int_{\Omega_r}\varphi(x,\nabla u_n)|\nabla u_n|^2\,dx.
\end{split}
\end{equation}
Now, we estimate the right hand side of \eqref{pos1}. Indeed, 
\begin{equation}\label{pos1rhs}
\begin{split}
\int_{\Omega_r}f(|x|,u_n)u_n\,dx&=\lambda\int_{\Omega_N}|u_n|^\beta e^{\alpha |u_n|^{N^\prime}}\,dx+\int_{\Omega_r\setminus\Omega_N}\lambda \eta_{N}(|x|)|u_n|^\beta e^{\alpha |u_n|^{N^\prime}}\,dx\\
&\quad+\int_{\Omega_q} g(|x|,u_n)u_n\,dx+\int_{\Omega_r\setminus\Omega_q}\tilde{\eta}(|x|)g(|x|,u_n)u_n\,dx\\
&\quad\quad+\int_{\Omega_p}|u_n|^{p^*}\,dx+\int_{\Omega_r\setminus\Omega_p}\eta_p(|x|)|u_n|^{p^*}\,dx\\
&\leq\lambda\int_{\Omega_N}|u_n|^\beta e^{\alpha |u_n|^{N^\prime}}\,dx+\int_{\Omega_q} g(|x|,u_n)u_n\,dx+\int_{\Omega_p}|u_n|^{p^*}\,dx\\
&\quad+\int_{(\overline{\Omega_q})_{\delta/2}} f(|x|,u_n)u_n\,dx\\
&=I_1+I_2+I_3+I_4.
\end{split}
\end{equation}

\noindent \textbf{Estimate of $I_1$:} As $||u_n|| \to 0$, by H\"older's inequality and Lemma 3.5 for some constant $C_1>0$ (independent of $n$), we have
\begin{equation}\label{estI1}
\begin{split}
I_1&=\lambda \int_{\Omega_N} |u_n|^\beta e^{\alpha |u_n|^{N^\prime}}\,dx\\
&\leq \lambda\left(\int_{\Omega_N}|u_n|^{2\beta}\,dx\right)^{\frac12}
\left(\int_{\Omega_N}e^{2\alpha|u_n|^{N^\prime}}\,dx\right)^{\frac12}\\
&\leq C_1||u_n||^{\beta}.
\end{split} 
\end{equation}

\noindent \textbf{Estimate of $I_2$:}
From the condition $(g_1)$ and the embedding \eqref{EMB1} for some constant $C_2>0$ (independent of $n$) we deduce that 

\begin{equation}\label{estI2}
\begin{split}
I_2&=\int_{\Omega_q}g(|x|,u_n)u_n dx\leq C_2 ||u_n||^{q_1}.
\end{split}
\end{equation}

\noindent \textbf{Estimate of $I_3$:} It is clear that 
\begin{equation}\label{estI3}
\begin{split}
I_3&=\int_{\Omega_p}|u_n|^{p^*} dx\leq C_3||u_n||^{p^*},
\end{split}
\end{equation}
for some constant $C_3>0$ (independent of $n$).

\noindent \textbf{Estimate of $I_4$:} By the embedding \eqref{EMB1}, the definition of $f_2$ and the condition $(g_1)$ we have that 
\begin{equation}\label{estI4}
\begin{split}
I_4&=\int_{(\overline{\Omega_q})_{\delta/2}}f(|x|,u_n)u_n dx\leq C_4(||u_n||^{q_1}+||u_n||^\beta+||u_n||^{p^*}),
\end{split}
\end{equation}
for some constant $C_4>0$ (independent of $n$).

Therefore, using the estimates \eqref{estI1}, \eqref{estI2}, \eqref{estI3} and \eqref{estI4} in \eqref{pos1rhs}, we obtain
\begin{equation}\label{pos1rhsest}
\int_{\Omega_r}f(|x|,u_n)u_n dx\leq C(||u_n||^{q_1}+||u_n||^\beta+||u_n||^{p^*}),
\end{equation}
for some constant $C>0$ (independent of $n$). Using \eqref{pos1rhsest} in \eqref{pos1lhs}, we have
\begin{equation}\label{posfinalest}
p||u_n||^q\leq p\int_{\Omega_r}\Phi(x,|\nabla u_n|)dx\leq C(||u_n||^{q_1}+||u_n||^\beta+||u_n||^{p^*}).
\end{equation}
Since all the parameters $\beta,q_1$ and $p^*$ are larger than $q$, from \eqref{posfinalest} for some constant $\widehat{C}>0$ (independent of $n$), we have
$$
||u_n||\geq \widehat{C},\text{ for all }u_n\in \mathcal{M}_{k,r},
$$
which is a contradiction to the fact $||u_n||\to 0$. Hence \eqref{pos} holds.

\noindent \textbf{Step 2.} From the definition of $I$, $(\varphi_3)$ and Proposition \ref{modular1} for any $u\in\mathcal{M}_{k,r}$, we have
\begin{align*}
I(u)&=I(u)-\frac{1}{\chi}I'(u)u\\
&=\int_{\Omega_r}\Phi(x,|\nabla u|)\,dx-\frac{1}{\chi}\int_{\Omega_r}\phi(x,|\nabla u_n|)|\nabla u_n|^2\,dx\\
&\geq \left(1-\frac{q}{\chi}\right)
\int_{\Omega_r}\Phi(x,|\nabla u|) dx\\
&\geq \left(1-\frac{q}{\chi}\right)\max\{||u||^q,||u||^p\}\\ 
&\geq \left(1-\frac{q}{\chi}\right)\max\{\eta^q,\eta^p\},
\end{align*}
where in the last line, we have used the estimate \eqref{pos} from Step 1 and $\chi=\min\{\theta,\beta,p^*\}$. This means that, 
$$
J_{k,r}\geq \left(1-\frac{q}{r}\right)\max\{\eta^q,\eta^p\}>0,
$$
for every $1\leq k\leq \infty$ and all $r>0$. Hence the result follows.
\end{proof}

\begin{lemma}\label{l42}
For any integer $1\leq k<\infty$, there exists $\lambda_0=\lambda_0(k)>0$, such that 
$$
J_{k,r}<M=\left(1-\frac{q}{\chi}\right)\min\left\{\frac1N\left(\frac{\alpha_N}{2^{N'}\alpha}\right)^{N-1},\frac{1}{p} {S_p^{N/p}}\right\},\text{ for all }\lambda\geq \lambda_0,\text{ and }\chi=\min\{\theta,\beta,p^*\}.
$$
\end{lemma}
\begin{proof}
Fix $1\leq k<\infty$. Due to $(\eta)$, there exists $\gamma=\gamma(k) < \min\left\{\frac{1}{2},\delta_1\right\}$ such that the ball $B_{\gamma,r}:=B_{\gamma}\Big(\big(\frac{2r+1}{2},0,\ldots,0\big)\Big)\subset \Omega_N\setminus{\overline{(\Omega_q)_\delta}}$
satisfies
$$
h^i B_{\gamma,r}\cap h^j B_{\gamma,r}=\emptyset,\text{ for all }h^i\in H_k, \, i\neq j\in\{0,1,\ldots,k-1\}.
$$
Consider $v_r\in W^{1,\Phi}_0(B_{\delta,r})\setminus\{0\}$ and define 
$$
v:=\sum_{h\in H_k} h v_r\in W^{1,\Phi}_{0,H_k}\left(\Omega_N\setminus{\overline{(\Omega_q)_\delta}}\right)\setminus\{0\}.
$$
By definition of $I$, we observe that
$$
I'(tv)tv=\int_{{\Omega_N\setminus{\overline{(\Omega_q)_\delta}}}}\varphi(x,|t\nabla v|) t^2|\nabla v|^2\,dx-\int_{\Omega_N\setminus{\overline{(\Omega_q)_\delta}}} f(x,tv)tv\,dx.
$$
Using 
$$
f(x,tv)tv\geq \lambda t^\beta|v|^\beta, \forall x\in \Omega_N\setminus{\overline{(\Omega_q)_\delta}}\text{ and }t\geq 0.
$$
Therefore, using the hypothesis $(\varphi_3)$ and (\ref{DES2}), for every $t\geq 0$, 
$$
I'(tv)tv\leq q\xi_1(t)\int_{\Omega_N\setminus{\overline{(\Omega_q)_\delta}}}\Phi(|\nabla v|)\,dx  -\lambda t^\beta \int_{\Omega_N\setminus{\overline{(\Omega_q)_\delta}}} |v|^{\beta}\,dx.
$$
As $\beta>q> p$, we get that $I'(tv)tv\to-\infty$ as $t\to+\infty$ and $I'(tv)tv>0$ for $t \approx 0$.

So, there exists $t_v>0$ such that $t_v v\in W^{1,\Phi}_{0,H_k}\left({\Omega_N\setminus{\overline{(\Omega_q)_\delta}}}\right)\setminus\{0\}$ with $I^\prime(t_v v)t_v v=0$. If we denote by $w=t_v v$, then 
\begin{equation}\label{max}
J_{k,r}\leq I(w)=kI(t_v v_r)=k\max_{t\geq0} I(t v_r).
\end{equation}

Following similar arguments as in the proof of \cite[Lemma 3.11]{AGR}, we have
$$
\max_{t\geq0} I(t_v v_r)\leq \frac{1}{\lambda^{\frac{N}{\beta-N}}}\left(\frac1N-\frac1\beta\right)\frac{\left(c_1||\nabla v_r||^N_{L^N(\Omega_N)}\right)^{\frac{\beta}{\beta-N}}}{\left(||v_r||^{\beta}_{L^\beta(\Omega_N)}\right)^{\frac{N}{\beta-N}}}.
$$
Now, we fix $\lambda_0=\lambda_0(k)>0$ such that for all $\lambda\geq \lambda_0$, we have 
\begin{equation}\label{lam0}
\frac{k}{\lambda^{\frac{N}{\beta-N}}}\left(\frac1N-\frac1\beta\right)\frac{\left(c_1||\nabla v_r||^N_{L^N(\Omega_N)}\right)^{\frac{\beta}{\beta-N}}}{\left(||v_r||^{\beta}_{L^\beta(\Omega_N)}\right)^{\frac{N}{\beta-N}}}<\left(1-\frac{q}{\chi}\right)\min\left\{\frac1N\left(\frac{\alpha_N}{2^{N'}\alpha}\right)^{N-1},\frac{1}{p} {S_p^{N/p}}\right\}.
\end{equation}
Therefore, from \eqref{max} and \eqref{lam0}, the result follows.
\end{proof}

\begin{lemma}\label{l43}
If $1\leq k<\infty$ and $\lambda\geq \lambda_0$, then $J_{k,r}$ is achieved.
\end{lemma}

\begin{proof}
Let $(v_n)\subset \mathcal{M}_{k,r}$ a minimizing sequence for $J_{k,r}$, i.e., $(v_n)\subset W^{1,\phi}_{0,H_k}(\Omega_r)\setminus\{0\}$ such that 
$$
I'(v_n)v_n=0\text{ and }I(v_n)\to J_{k,r}.
$$
We claim that $(v_n)$ is bounded. Assume there is some $n$ such that $||v_n||\geq 1$, since otherwise $(v_n)$ is bounded. Due to the fact $I(v_n)\to J_{k,r}$ and $J_{k,r} \leq M$, where $M$ was given in Lemma \ref{l42}, it follows that  
\begin{align*}
	M&\geq I(v_n)=I(v_n)-\frac{1}{\chi}I^\prime(v_n)v_n\\
	&\geq\left(1-\frac{q}{\chi}\right)||v_n||^p.
\end{align*}
Therefore, $||v_n||\leq c$ if $\|v_n\| >1$, for some constant $c>0$ independent of $n$. This shows that $(v_n)$ is bounded. 

\textbf{Claim:} 
$$
I'(v_n)\to 0\text{ in }(W^{1,\Phi}_{0,H_k}(\Omega_r))'.
$$

Indeed, using the Ekeland variational Principle (see Willem \cite{Willem}), there exists a sequence $(w_n)\subset \mathcal{M}_{k,r}$ such that 
$$
w_n=v_n+o_n(1),\, I(w_n)\to J_{k,r}
$$
and
\begin{equation}\label{4.2}
I'(w_n)-\ell_n E'(w_n)=o_n(1),
\end{equation}
where $(\ell_n)\subset \mathbb{R}$ and $E(w)=I'(w)w$ for $w\in W^{1,\Phi}_{0,H_k}(\Omega_r)$. Since $(v_n)$ is bounded, we also have that $(w_n)$ is bounded. Now, we prove that there exists $C >0$ such that 
\begin{equation}\label{uni}
|E'(w_n)  w_n|>C\text{ for all } n\in\mathbb{N}.
\end{equation}
Indeed, we observe that
\begin{equation}\label{unipos}
\begin{split}
-E^\prime(w_n)w_n&=-\int_{\Omega_r}\left[\frac{\partial}{\partial t}\varphi(x,|\nabla w_n|)|\nabla w_n|+2\varphi(x,|\nabla w_n|)\right]|\nabla w_n|^2\,dx\\
&\quad\quad+\int_{\Omega_r}\left[  \frac{\partial}{\partial t}f(|x|,w_n)w_n^{2}+f(|x|,w_n)w_n  \right]\,dx\\
&\geq -q\int_{\Omega_r}\varphi(x,|\nabla w_n|)|\nabla w_n|^2\,dx\\
&\quad\quad+\int_{\Omega_r}\left[\frac{\partial}{\partial t}f(|x|,w_n)w_n^{2}+f(|x|,w_n)w_n\right]\,dx\\
&=\int_{\Omega_r}\left[\frac{\partial}{\partial t}f(|x|,w_n)w_n^{2}-(q-1)f(|x|,w_n)w_n\right]\,dx,
\end{split}
\end{equation}
where in the second step, we have used the hypothesis $(\varphi_{9})$ and in the third step, the property $I^\prime(w_n)w_n=0$, i.e. 
$$
\int_{\Omega_r}\varphi(x,|\nabla w_n|)|\nabla w_n|^2\,dx=\int_{\Omega_r}f(|x|,w_n)w_n\,dx,
$$
respectively. Since $(w_n)$ is bounded in $W_{0}^{1,\Phi}(\Omega_r)$ and $J_{k,r}>0$, there exists $w\in W^{1,\Phi}(\Omega_r)\setminus\{0\}$ such that  $w_n\to w$ strongly in $L^{\Phi}(\Omega_r)$ and $w_n(x) \to w(x)$ a.e. in $\Omega_r$ for some subsequence. Then, by Fatou's lemma,
\begin{equation}\label{Fatous}
\begin{split}
&\liminf_{n\to\infty}\int_{\Omega_r}\left[\frac{\partial}{\partial t}f(|x|,w_n)w_n^{2}-(q-1)f(|x|,w_n)w_n\right]\,dx\\
&\geq\int_{\Omega_r}\left[\frac{\partial}{\partial t}f(|x|,w)w^{2}-(q-1)f(|x|,w)w\right]\,dx.
\end{split}
\end{equation}
From the hypothesis $(g_4)$ and the definition of $f_2$, we obtain
\begin{equation}\label{g3f}
\int_{\Omega_r}\left[\frac{\partial}{\partial t}f(|x|,w)w^{2}-(q-1)f(|x|,w)w\right]\,dx>0.
\end{equation}
By contradiction, suppose
$$
\lim_{n\to\infty}E^\prime(w_n)w_n=0.
$$
Then letting $n\to\infty$ in \eqref{unipos} and using \eqref{Fatous} along with \eqref{g3f}, we obtain
\begin{equation*}
0=\lim_{n\to\infty}E^\prime(w_n)w_n\geq\int_{\Omega_r}\left[\frac{\partial}{\partial t}f(|x|,w)w^{2}-(q-1)f(|x|,w)w\right]\,dx>0, 
\end{equation*}
which is absurd. Therefore \eqref{uni} holds. 

From, (\ref{4.2}) $$ \ell_n E'(w_n)w_n=o_n(1),$$
and so, $\ell_n=o_n(1)$. Since $(w_n)$ is bounded we get $\big(E'(w_n)\big)$ is bounded. Hence from (\ref{4.2}) 
$$
I'(w_n)\to 0\text{ in }\Big(W^{1,\Phi}_{0,H_k}(\Omega_r)\Big)'.
$$
Thus, without loss generality, we may assume 
$$
I(v_n)\to J_{k,r}\;\; \text{and} \;\; I'(v_n)\to 0.
$$

Since $(v_n)$ is bounded, there exists $v\in W^{1,\Phi}_{0.H_k}(\Omega_r)$ such that, for a subsequence we have
$$
\begin{cases}
\begin{array}{llll}
&v_n\rightharpoonup  v \quad &\text{in} & \, W^{1,\Phi}_{0,H_k}(\Omega_r),\\
&v_n(x)\to v(x),\quad &  \text{a.e. in} &\, \Omega_r.
\end{array}
\end{cases}
$$
Now following exactly the proof of Lemma \ref{PS}, we get $I(v_n)\to I(v)=J_{k,r}$. Hence, $J_{k,r}$ is achieved.
\end{proof}

Now we establish the following Strauss-type result in Musielak-Sobolev space, which would be very useful to find a lower bound of $J_{\infty,r}$.

\begin{lemma}\label{strauss} (A Strauss-type result in Musielak-Sobolev space)
Assume that $(\varphi_1)$, $(\varphi_2)$, $(\varphi_3)$, $(\varphi_{10})-(\varphi_{11})$ 
holds and let $v\in W^{1,\Phi}(\mathbb{R}^N)$ be a radial function. Then 
$$
|v(x)|\leq \Phi^{-1}\left(x,\frac{C}{|x|^{N-1}}\int_{\mathbb{R}^N}\big[\Phi(x,|v|)+\Phi(x,|\nabla v|)\big]\,dx\right) \text{ a.e. in }\mathbb{R}^N,
$$
where $\Phi^{-1}(x,\cdot)$ denotes the inverse function of $\Phi(x,\cdot)$ restricted to $[0,+\infty)$ and $C$
is a positive constant independent of $v$.
\end{lemma}
 
 \begin{proof}
 We will establish the result for radial functions in $C_{0}^{\infty}(\mathbb{R}^N)$.
 
 Let $v\in C_0^{\infty}(\mathbb{R}^N)$ be radial and let $|x|=r$, $w(r)=v(x)$. Then, from $(\varphi_{10})$
 $$
 \Phi\big(b,w(b)\big)-\Phi\big(r,w(r)\big)=\int_r^b \left(\frac{d}{ds} \Phi\big(s,w(s)\big)\right)\,ds,\text{ for all }b>r>0.
 $$
 
 Since $w\in C_0^{\infty}([0,\infty))$, for $b$ large enough,
\begin{equation}\label{rad}
\begin{split}
\Phi\big(r,w(r)\big)&=-\int_r^{\infty}\frac{\partial}{\partial s}\Phi\big(s, w(s)\big)\,ds-\int_r^{\infty}\varphi\big(s,w(s)\big)w(s)w^\prime(s)\,ds\\
&\leq\int_r^{\infty}\Big|\frac{\partial}{\partial s}\Phi\big(s, w(s)\big)\Big|\, ds+\int_r^{\infty}\varphi\big(s,|w(s)|\big) |w(s)| |w'(s)|\,ds.
\end{split}
\end{equation}

 From (2.1), (2.2) and the $\Delta_2$ condition \eqref{delta2}, for all $s\geq 0$,
 \begin{equation}\label{rad1}
 \begin{split}
 \varphi\big(s,|w(s)|\big) |w(s)| |w^\prime(s)|&\leq \tilde{\Phi}\Big(s,\varphi\big(s,|w(s)|\big)|w(s)|\Big)+\Phi\big(s,|w'(s)|\big)\\
 &\leq \Phi\big(s,2|w(s)|\big)+\Phi\big(x,|w^\prime(s)|\big)\\
 &\leq K\Phi\big(s,|w(s)|\big)+\Phi\big(s,|w^\prime(s)|\big).
 \end{split}
 \end{equation}
 
 From $(\varphi_{11})$, for all $s\geq 0$,
 \begin{equation}\label{rad2}
 \begin{split}
 \Big|  \frac{\partial}{\partial s} \Phi\big(s,w(s)\big)  \Big| \leq M \Phi \big(s,w(s)\big).
 \end{split}
 \end{equation}
 
 Now using \eqref{rad1} and \eqref{rad2} in \eqref{rad}, we obtain
$$
\Phi\big(r,w(r)\big)\leq (M+K+1)\int_r^{\infty}\big[\Phi\big(s,|w(s)|\big)+\Phi\big(s,|w'(s)|\big)\big] \,ds.
$$
Hence, we can conclude that 
$$
\Phi\big(r,w(r)\big)\leq \frac{(M+K+1)}{r^{N-1}}\int_r^{+\infty}\big[\Phi\big(s,|w(s)|\big)+\Phi\big(s,|w'(s)|\big)\big] s^{N-1}\,ds.
$$
From this, there is $C>0$ such that 
$$
\Phi\big(x,v(x)\big)\leq \frac{C}{|x|^{N-1}}\int_{\mathbb{R}^N}\big[\Phi\big(x,|v|\big)+\Phi\big(x,|\nabla v|\big)\big]\,dx.
$$
Since $\Phi$ is an even function, $\Phi\big(x,v(x)\big)=\Phi\big(x,|v(x)|\big)$ for all $x\in\mathbb{R}^N$ and so, 
$$
\Phi\big(x,|v(x)|\big)\leq \frac{C}{|x|^{N-1}}\int_{\mathbb{R}^N}\big[\Phi\big(x,|v|\big)+\Phi\big(x,|\nabla v|\big)\big]\,dx.
$$
From this,  $$|v(x)|\leq \Phi^{-1}\left(x,\left(\frac{C}{|x|^{N-1}}\int_{\mathbb{R}^N}\big[\Phi\big(x,|v|\big)+\Phi\big(x,|\nabla v|\big)\big]\,dx\right)\right),\text{ for all }x\in \mathbb{R}^N\setminus\{0\},
$$
where
$\Phi^{-1}(x, \cdot)$ denotes the inverse function of $\Phi(x, \cdot)$ restricted to $[0,+\infty)$. Now the result follows from the density of $C_0^{\infty}(\mathbb{R}^N)$ in $W^{1,\Phi}(\mathbb{R}^N)$, because $\Phi$ satisfies the $\Delta_2$ condition. 
\end{proof}

\begin{lemma}\label{lemma0.4}
There exists $r_0=r_0(\lambda)>0$ such that for $\chi=\min\{\theta,\beta,p^*\}$,
$$
J_{\infty,r}\geq \left(1-\frac{q}{\chi}\right)\min\left\{\frac1N\left(\frac{\alpha_N}{2^{N'}\alpha}\right)^{N-1},\frac{1}{p} {S_p^{N/p}}\right\},\text{ for all }r>r_0.
$$
\end{lemma}

\begin{proof}
By contradiction, suppose there exists a sequence $(r_n)$ such that $r_n\to\infty$, satisfying
\begin{equation}\label{4.13}
J_{\infty,r_n}<\left(1-\frac{q}{\chi}\right)\min\left\{\frac1N\left(\frac{\alpha_N}{2^{N'}\alpha}\right)^{N-1},\frac{1}{p} {S_p^{N/p}}\right\},\text{ for all }n\in\mathbb{N}.
\end{equation}

First, we claim that $J_{\infty,r_n}$ is attained, for all $n\in\mathbb{N}$. In fact, for a fixed $n$, let $(v_k)\subset \mathcal{M}_{\infty,r_n}$ be a minimizing sequence for $J_{\infty,r_n}$, i.e., $(v_k)\subset W^{1,\Phi}_{0,H_{\infty}}(\Omega_{r_n})\setminus\{0\}$ and satisfies
$$
I'(v_k)v_k=0,\text{ and }I(v_k)\to J_{\infty,r_n},\text{ as }k\to\infty.
$$

Note that 
\begin{equation}\label{4.14}
\begin{split}
o_k(1)+J_{\infty,r_n}&=I(v_k)-\frac{1}{\chi}I'(v_k)v_k\\
&\geq \left(1-\frac{q}{\chi}\right)\int_{\Omega_{r_n}}\Phi(x,|\nabla v_k|)\,dx\\
&\geq \frac{1}{N}\left(1-\frac{q}{\chi}\right)\int_{\Omega_N}|\nabla v_k|^N\,dx.
\end{split}
\end{equation}
Using \eqref{4.13} in \eqref{4.14},
$$
\limsup_{k\to+\infty }||\nabla v_k||_{W^{1,N}(\Omega_N)}^N <\left(\frac{\alpha_N}{2^{N'}\alpha}\right)^{N-1}.
$$
Now, we can repeat the same arguments employed in the proof of Lemma \ref{l43}  to conclude that
$$
I'(v_k)\to 0\,\text{in}\, (W^{1,\Phi}_{0,H_{\infty}}(\Omega_{r_n}))'\text{ and }v_k\to v\text{ in } W^{1,\Phi}_{0,H_{\infty}}(\Omega_{r_n})
$$
where $v\in W^{1,\Phi}_{0,H_{\infty}}(\Omega_{r_n})$ is the limit of $(v_k)$ in $W^{1,\Phi}_{0,H_{\infty}}(\Omega_{r_n})$. Then,
$$
I(v_k)\to I(v)=J_{\infty,r_n}\,\text{and}\, I'(v_k)\to I'(v)=0.
$$

Hence $I(v)=J_{\infty,r_n}$. Note that $v\neq 0$, since by Lemma \ref{pos} $J_{\infty,r_n}>0$. Therefore $v\in \mathcal{M}_{\infty,r_n}$ and  $J_{\infty,r_n}$ is attained at $v$.

Therefore, for each $n\in\mathbb{N}$ we can choose a sequence $\{u_n\}\subset W^{1,\Phi}_{0,H_\infty}(\Omega_{r_n})\setminus\{0\}$ satisfying
$$
I'(u_n)u_n=0\;\;\; \text{and} \;\;\;  I(u_n)=J_{\infty,r_n}.
$$
Proceeding as in \eqref{4.14} 
$$
\frac{1}{N}\left(1-\frac{q}{\chi}\right)\left(\frac{\alpha_N}{2^{N'}\alpha}\right)^{N-1}>J_{\infty,r_n}=I(u_n)-\frac{1}{\chi} I'(u_n)u_n\geq \frac{1}{N}\left(1-\frac{q}{\chi}\right)\int_{\Omega_N}|\nabla u_n|^N\,dx
$$
which implies 
\begin{equation}\label{4.15}
\limsup_{k\to+\infty }||\nabla u_n||_{W^{1,N}(\Omega_N)}^N <\left(\frac{\alpha_N}{2^{N'}\alpha}\right)^{N-1}.
\end{equation}
Let $\{\tilde{u}_n\}$ be a sequence given by 
$$ 
\tilde{u}_n(x)=\begin{cases} u_n(x),\text{ if }x\in\Omega_{r_n},\\ 0,\,\,\qquad \text{if}\, x\notin \Omega_{r_n}. \end{cases}$$
Observe that the following properties hold:
\begin{itemize}
\item[(1)] $\{\tilde{u}_n\}\subset W^{1,\Phi}_{H_{\infty}}(\mathbb{R}^N)$;
\item[(2)] $||\tilde{u}_n||_{W^{1,\Phi}_{H_{\infty}}(\mathbb{R}^N)}=||u_n||_{W^{1,N}_{0,H_{\infty}}(\Omega_{r_n})}$;
\item[(3)] $\tilde{u}_n \rightharpoonup 0\,\text{in}\, W^{1,\Phi}_{H_{\infty}}(\mathbb{R}^N)\text{ because }\tilde{u}_n(x)\to 0 \text{ a.e. in }\mathbb{R}^N$.
\end{itemize}
Therefore, we have
\begin{equation}\label{tilde}
\int_{\mathbb{R}^N}\varphi(x,|\nabla\tilde{u}_n|)|\nabla\tilde{u}_n|^2\,dx=\int_{\mathbb{R}^N}f(x,\tilde{u}_n)\tilde{u}_n\,dx.
\end{equation}
From Lemma \ref{strauss} we deduce that the sequence $\{\tilde{u}_n\}$ satisfies 
$$
|\tilde{u}_n|_\infty \to 0.
$$
Using the fact that 
$$
\lim_{t \to 0}\frac{f(x,t)t}{\varphi(x,t)t^2}=0, \quad \forall\, x \in \mathbb{R}^N,
$$
and $(\varphi_3)$, it follows that given  $\epsilon < \frac{p}{q \Upsilon}$, where $\Upsilon$ was given in Lemma \ref{poincare},  there exists $n_0 \in\mathbb{N}$ such that 
$$
\int_{\mathbb{R}^N}f(x,\tilde{u}_n)\tilde{u}_n\,dx \leq \epsilon \int_{\mathbb{R}^N}\varphi(x,|\tilde{u}_n|)|\tilde{u}_n|^2\,dx \leq \varepsilon q \int_{\mathbb{R}^N} \Phi(|\tilde{u}_n|)\,\, dx=\varepsilon q \int_{\Omega_{r_n}} \Phi(|\tilde{u}_n|)\,\, dx, \, n \geq n_0.
$$
Since $r_n \to +\infty$, without loss of generality we can assume that $r_n \geq 1 $ for all $n \in \mathbb{N}$. Therefore, by the Poincar\'e inequality from Lemma \ref{poincare}, 
\begin{equation} \label{NEWEQUATION}
\int_{\mathbb{R}^N}f(x,\tilde{u}_n)\tilde{u}_n\,dx \leq \epsilon q \Upsilon \int_{\Omega_{r_n}} \Phi(|\nabla \tilde{u}_n|)\,\, dx, \quad \forall n \geq n_0.	
\end{equation}	
From (\ref{tilde}), (\ref{NEWEQUATION}) and $(\varphi_3)$ 
\begin{align*}
p\int_{\Omega_{r_n}}\Phi(|\nabla \tilde{u}_n|)\,dx&=p\int_{\mathbb{R}^N}\Phi(|\nabla \tilde{u}_n|)\,dx\\ &\leq \int_{\mathbb{R}^N}\varphi(x,|\nabla\tilde{u}_n|)|\nabla\tilde{u}_n|^2\,dx\\
&\leq \epsilon q \Upsilon \int_{\Omega_{r_n}} \Phi(|\nabla \tilde{u}_n|)\,\, dx,\quad \forall n \geq n_0.
\end{align*}
As $\tilde{u}_n \not=0$ for all $n \in \mathbb{N}$, we get $p \leq \epsilon q \Upsilon$, 
which is absurd. Hence, the result follows.

\end{proof}

\begin{lemma}\label{lemma0.6}
Suppose that $J_{km,r}$ is attained for some $1\leq k<\infty$ and some $2\leq m<\infty$. Suppose also that $J_{km,r}<J_{\infty,r}$. Then, $J_{k,r}<J_{km,r}$. 
\end{lemma}

\begin{proof}
Consider $u\in\mathcal{M}_{km,r}$ such that $I(u)=J_{km,r}$. Let $x=(\theta,\rho)$ be the polar coordinates of $x\in\mathbb{R}^2$. Then, $u=u(\theta,\rho,|y|),\,y\in \mathbb{R}^{N-2}$. Note that 
$$
\Phi\big(\sqrt{\rho^2+|y|^2},|\nabla u|\big)=\Phi\left(\sqrt{\rho^2+|y|^2},\left(\frac{1}{\rho^2}u_\theta^2+u_\rho^2+|\nabla_y u|^2\right)^{1/2}\right).
$$
Define
$$
v(\theta,\rho,|y|):=u\left(\frac{\theta}{m},\rho,|y|\right),
$$
We observe that,
\begin{itemize}
\item[(i)] $v\in W^{1,\Phi}_{0,H_k}(\Omega_r)$;
\item[(ii)]$\Phi\big(\sqrt{\rho^2+|y|^2},|\nabla v|\big)=\Phi\left(\sqrt{\rho^2+|y|^2},\left(\frac{1}{\rho^2m^2}u_\theta^2+u_\rho^2+|\nabla_y u|^2\right)^{1/2}\right)$;
\item[(iii)] $\int_{\Omega_r}F(|z|,v)\,dx dy=\int_{\Omega_r}F(|z|,u)\,dx dy$, where $z=(x,y) \in \Omega_r$.
\end{itemize}
Proceeding similarly as in the proof of Lemma \ref{l42}, there exists $t_0>0$ such that $t_0v\in\mathcal{M}_{k,r}$. For simplicity, let $v:=t_0 v$. Now, since $v\in\mathcal{M}_{k,r}$, 
$$
J_{k,r}\leq I(v)=\int_{\Omega_r}\Phi(|z|,|\nabla v|)\,dx dy-\int_{\Omega_r} F(|z|,v)\,dx dy.
$$
Using $(ii)-(iii)$, 
\begin{equation}\label{e419}
J_{k,r}\leq \int_{\Omega_r}\Phi\left(\sqrt{\rho^2+|y|^2},\left(\frac{1}{\rho^2m^2}u_\theta^2+u_\rho^2+|\nabla_y u|^2\right)^{1/2}\right)\,dx dy-\int_{\Omega_r}F(|x|,u)\,dxdy.
\end{equation}
If $I(u)=J_{km,r}<J_{\infty,r}$, then $u\notin W^{1,\Phi}_{0,H_{\infty}}(\Omega_r)$, then $u^2_{\theta}$ is not identically zero. Therefore, using that $m>1$,
\begin{equation*}
\begin{split}
&\int_{\Omega_r}\Phi\left(\sqrt{\rho^2+|y|^2},\left(\frac{1}{\rho^2m^2}u_\theta^2+u_\rho^2+|\nabla_y u|^2\right)^{1/2}\right)\,dx
dy\\
&<\int_{\Omega_r}\Phi\left(\sqrt{\rho^2+|y|^2},\left(\frac{1}{\rho^2}u_\theta^2+u_\rho^2+|\nabla_y u|^2\right)^{1/2}\right)\,dxdy,
\end{split}
\end{equation*}
which together with $(\ref{e419})$  implies $J_{k,r}<I(u)=J_{km,r}$ and the proof is complete.

\end{proof}

\section{Proof of the main results}

\subsection{Proof of Theorem \ref{T1}}
By our assumption on $f_1$, it is easy to verify that $I$ is even and $I(0)=0$. Moreover, from Lemma \ref{L1}, Lemma \ref{LMP} and Lemma \ref{PS}, $I$ satisfies all the properties in Lemma \ref{AmRbthm}. Now from Lemma \ref{level} we have there exists positive real numbers $\lambda_m,\mu_m,\tau_m$ such that for all $\lambda\geq\lambda_m,\mu\geq\mu_m$ and $\tau\geq\tau_m,$ we have 
$$
0<c_1^{\lambda,\mu,\tau}\leq c_2^{\lambda,\mu,\tau}\leq\cdots\leq c_m^{\lambda,\mu,\tau}<M.
$$
Hence by Lemma \ref{AmRbthm}, for every $\lambda\geq\lambda_m,\mu\geq\mu_m$ and $\tau\geq\tau_m,$ the problem $(P)$ has at least $m$ pairs of nontrivial solutions.
\qed

\subsection{Proof of Theorem \ref{T2}}
By Lemma \ref{l42}, for each $n\in\mathbb{N}$, there exists $\lambda_0=\lambda_0(n)>0$ satisfying
$$J_{2^n,r}<\left(1-\frac{q}{\chi}\right)\min\left\{\frac1N\left(\frac{\alpha_N}{2^{N'}\alpha}\right)^{N-1},\frac{1}{p} {S_p^{N/p}}\right\},\quad\forall\lambda\geq \lambda_0.$$ 
On the other hand, by Lemma \ref{lemma0.4}, there exists $r_0=r_0(\lambda_0,n)>0$ such that 
$$J_{\infty,r}\geq \frac{1}{N}\left(1-\frac{q}{\chi}\right)\left(\frac{\alpha_N}{2^{N'}\alpha}\right)^{N-1},\quad\forall r>r_0.$$
Thus, 
$$
0<J_{2^n,r}=J_{2.2^{n-1},r}<\left(1-\frac{q}{\chi}\right)\min\left\{\frac1N\left(\frac{\alpha_N}{2^{N'}\alpha}\right)^{N-1},\frac{1}{p} {S_p^{N/p}}\right\}\leq J_{\infty,r},
$$
for all $\lambda>\lambda_0$ and for all $r>r_0$. By Lemma \ref{l43}, we have $J_{2^n,r}$ is attained. So, we can apply Lemma \ref{lemma0.6} to get
$$
J_{2^{n-1},r}<J_{2^n,r},\text{ for all }\lambda>\lambda_0\text{ and for all }r>r_0.
$$
Now $J_{2^{n-2}2,r}$ is also attained and satisfies
$$
J_{2^{n-2}2,r}=J_{2^{n-1},r}<J_{2^{n},r}<J_{\infty,r}.
$$
Again, by Lemma \ref{lemma0.6}, we get $J_{2^{n-2},r}<J_{2^{n-1},r}$. Inductively, 
$$
0<J_{2,r}<J_{2^2,r}<\ldots<J_{2^n,r}<J_{\infty,r},
$$
for all $\lambda>\lambda_0$ and all $r>r_0$.

Noting Lemma \ref{L1}, from Lemma \ref{l43}, minimizers of $J_{k,m}$ are critical points of $I$ in $W^{1,\Phi}_{0,H_k}(\Omega_r)$. Now, applying the Principle of symmetric criticality from \cite{Squassina}, it follows that they are critical points of $I$ in $W^{1,\Phi}_0(\Omega_r)$ and therefore are solutions of $(P)$. Note that, due to Lemma \ref{pos}, such solutions are nontrivial. Therefore, all minimizers of $J_{2^m,r},\, m=1,\ldots,n$ are nonradial, rotationally nonequivalent and nontrivial solutions of $(P)$. \qed

\end{document}